\documentclass[letterpaper, 10 pt, journal]{IEEEtran}
\IEEEoverridecommandlockouts
\usepackage[english]{babel}

\usepackage{amsmath, amssymb, amsthm}
\usepackage{cite}
\usepackage{xcolor}
\usepackage{stmaryrd}
\usepackage{threeparttable}
\usepackage{booktabs}
\usepackage[caption=false]{subfig}

\usepackage[linesnumbered,ruled,vlined]{algorithm2e}
\usepackage{optidef}

\usepackage{enumitem}
\usepackage{url}
\usepackage{accents}
\usepackage[version=4]{mhchem}
\usepackage[makeroom]{cancel}
\usepackage{soul} 

\newcommand{\norm}[1]{\left\lVert#1\right\rVert}

\newtheorem{theorem}{Theorem}

\newtheorem{lemma}{Lemma}
\newtheorem{assumption}{Assumption}

\newtheorem{remark}{Remark}

\def\matt#1{\begin{bmatrix}#1\end{bmatrix}}

\def\BibTeX{{\rm B\kern-.05em{\sc i\kern-.025em b}\kern-.08em
    T\kern-.1667em\lower.7ex\hbox{E}\kern-.125emX}}

\def\diag{{\tt diag}}

\usepackage{tabularx}

\usepackage[hidelinks]{hyperref}
\hypersetup{
    colorlinks=true,
    linkcolor=blue,
    filecolor=magenta,
    urlcolor=cyan,
    citecolor=red
}

\begin{document}

\title{
Distributed Optimization under Edge Agreement with Application in Battery Network Management
}

\author{Zehui Lu, \ Shaoshuai Mou
\thanks{The authors are with the School of Aeronautics and Astronautics, Purdue University, IN 47907, USA {\tt\small \{lu846,mous\}@purdue.edu} }
\thanks{This research is supported in part by the NASA University Leadership Initiative (ULI) project (grant no. 80NSSC20M0161), the NSF project (grant no. 2120430), and a gift funding from Northrop Grumman Corporation.}
}

\maketitle

\begin{abstract}
This paper investigates a distributed optimization problem under edge agreements, where each agent in the network is also subject to local convex constraints. Generalized from the concept of consensus, a group of edge agreements represents the constraints defined for neighboring agents, with each pair of neighboring agents required to satisfy one edge agreement constraint.
Edge agreements are defined locally to allow more flexibility than a global consensus, enabling heterogeneous coordination within the network.
This paper proposes a discrete-time algorithm to solve such problems, providing a theoretical analysis to prove its convergence. Additionally, this paper illustrates the connection between the theory of distributed optimization under edge agreements and distributed model predictive control through a distributed battery network energy management problem. This approach enables a new perspective to formulate and solve network control and optimization problems.
\end{abstract}


\section{Introduction} \label{sec:intro}

Recent research has seen a growing interest in distributed algorithms for networked multi-agent systems (MAS).
These algorithms aim to achieve global objectives through local coordination among network agents.
The concept of consensus, where all agents agree on a particular quantity, has become a cornerstone in developing distributed algorithms for MAS \cite{cao2008agreeing}.
Key research areas include multi-agent formation control \cite{chen2019controlling, rai2024safe}, multi-agent optimal control \cite{lu2022cooperative}, distributed computation \cite{notarnicola2017distributed,wang2019distributed}, and distributed optimization \cite{nedic2010constrained,tang2020distributed,srivastava2021network,chen2023differentially,rikos2024distributed}.
Among these topics, distributed optimization has received significant attention due to its ability to solve network control or optimization problems, such as those in power grids \cite{xie2019distributed, patari2022distributed}, battery energy storage systems (BESS) \cite{fang2016cooperative, farakhor2023scalable}, robot co-design \cite{lu2024distributed_co_design}, and machine learning \cite{koloskova2021improved}.
In these problems, each agent or node knows its local objective function and constraints, aiming to minimize the sum of all local objective functions with only local information and neighboring information while satisfying all local constraints and achieving consensus on decision variables.

Despite the development of numerous distributed algorithms based on consensus, they are typically designed for scenarios where all agents must reach the same value regarding a specific quantity.
However, practical applications often require heterogeneous coordination among agents beyond simple consensus, necessitating edge-dependent constraints across the entire network. Reference \cite{lu2024distributed} introduces the concept of edge agreements, constraints defined for neighboring agents, with each pair of neighboring agents corresponding to one such constraint.
While global consensus is a special case of edge agreements, the latter allows for more flexibility and can handle heterogeneous coordination among neighboring agents.
This flexibility facilitates solving the aforementioned network control or optimization problem by distributed model predictive control (MPC) \cite{mota2014distributed, negenborn2014distributed, houska2022distributed}, where the edge agreements can represent the dynamic equality constraint and coordination between every two neighboring agents.
Reference \cite{lu2024distributed} proposes a continuous-time algorithm for solving distributed optimization under edge agreements.
However, this algorithm does not account for additional constraints on the decision variables, which limits its practical application in distributed Model Predictive Control (MPC) \cite{negenborn2014distributed, houska2022distributed}.
In distributed MPC, it is essential to satisfy state and control constraints, such as maintaining intermediate states, controls, and terminal states within convex sets for each agent \cite{houska2022distributed,arauz2022cyber}.
These requirements impose additional convex or linear constraints on each agent's decision variable, which are not addressed in \cite{lu2024distributed}.

Therefore, this paper aims to develop a discrete-time distributed algorithm for solving distributed optimization problems under edge agreements, where each agent is also subject to local convex constraints. Additionally, this paper seeks to bridge the theory of distributed optimization under edge agreements with distributed MPC, exemplified by solving a distributed battery network energy management problem.

This paper is organized as follows.
Section~\ref{sec:problem_formulation} formulates the problem of interest mathematically and introduces some necessary assumptions and notations.
Section~\ref{sec:main_result} presents the theoretical results, including the proposed distributed algorithm, the convergence analysis, and a numerical example.
Section~\ref{sec:app} introduces a distributed battery network energy management problem, and presents how to apply the proposed algorithm to solve this problem with some simulation results.
Section~\ref{sec:conclusion} concludes this paper.

\textbf{\emph{Notations.}}  
Let $\llbracket a,b \rrbracket$ denote a set of all integers between $a \in \mathbb{Z}$ and $b \in \mathbb{Z}$, with both ends included.
Let $\mathbb{Z}_+$ denote the positive integer set.
Let $\vert \mathcal{N} \vert$ denote the cardinality of a set $\mathcal{N}$.
Let $\norm{\cdot}$ denote the Euclidean norm.
Let $\otimes$ denote the Kronecker product.
Let $\text{col}\{ \boldsymbol{v}_1, \cdots, \boldsymbol{v}_a \}$ denote a column stack of elements $\boldsymbol{v}_1, \cdots, \boldsymbol{v}_a $, which may be scalars, vectors or matrices,
i.e. $\text{col}\{ \boldsymbol{v}_1, \cdots, \boldsymbol{v}_a \} \triangleq {\matt{{\boldsymbol{v}_1}^{\top} & \cdots & {\boldsymbol{v}_a}^{\top}}}^{\top}$.
Let $\boldsymbol{0}_{m\times n}$ and $\boldsymbol{1}_{m\times n}$ be a matrix in $\mathbb{R}^{m \times n}$ with all zeros and ones, respectively; simplify the notation as $\boldsymbol{0}_{m}$ and $\boldsymbol{1}_{m}$ when $n=1$.
Let $\boldsymbol{I}_n \in \mathbb{R}^{n \times n}$ denote an identity matrix in $\mathbb{R}^{n \times n}$.


\section{Problem Formulation} \label{sec:problem_formulation}

Consider a networked multi-agent system consisting of $m$ agents labeled as $\mathcal{V}=\{1,\cdots,m\}$. Each agent $i$ can update its state $\boldsymbol{x}_i \in \mathbb{R}^n$ at every discrete time $k$ given bidirectional communication within its nearby neighbors denoted by a set $\mathcal{N}_i$.
Denote agent $i$'s state at time $k$ as $\boldsymbol{x}_{i,k}$.
Here, assume $i \notin \mathcal{N}_i$. Let $\mathbb{G} = \{ \mathcal{V}, \mathcal{E} \}$ denote the undirected graph such that an undirected edge $(i,j) \in \mathcal{E} $ if and only if agent $i$ and agent $j$ are neighbors. Let $\bar{m} \triangleq  \vert \mathcal{E} \vert $ the number of edges in $\mathbb{G}$. Suppose each agent $i$ only knows its private objective function $\boldsymbol{f}_i: \mathbb{R}^n \mapsto \mathbb{R}$.

The problem of interest to develop a discrete-time update rule for each agent $i$ to update $\boldsymbol{x}_i \in \mathcal{X}_i$ such that each $\boldsymbol{x}_i$ converges to a constant vector, which minimizes the global sum of local objective functions $\sum_{i=1}^{m} \boldsymbol{f}_i$ and satisfies the edge agreement constraint, i.e.
\begin{mini!}|s|
{{\{\boldsymbol{x}_i\}}_{i=1}^m}{ \textstyle\sum_{i=1}^{m} \boldsymbol{f}_i(\boldsymbol{x}_i) \label{problem_interest:obj}}
{\label{problem_interest}}{}
\addConstraint{ \boldsymbol{x}_i \in \mathcal{X}_i, \ \forall i \in \mathcal{V} \label{problem_interest:set}}
\addConstraint{ \boldsymbol{A}_{ij}(\boldsymbol{x}_i-\boldsymbol{x}_j)=\boldsymbol{b}_{ij}, \ \forall (i,j) \in \mathcal{E}. \label{problem_interest:edge}}
\end{mini!}
Here, $\mathcal{X}_i \subset \mathbb{R}^n$ is a nonempty closed convex set; $\boldsymbol{A}_{ij} \in \mathbb{R}^{d_{ij} \times n}$ and $\boldsymbol{b}_{ij} \in \mathbb{R}^{d_{ij}}$ are constant matrices, and privately known to agent $i$;
$d_{ij}$ is the dimension of edge agreement associated with edge $(i,j)$;
$\boldsymbol{f}_i: \mathbb{R}^n \mapsto \mathbb{R}$ is privately known to agent $i$. Denote $\boldsymbol{f}(\boldsymbol{x}) \triangleq \sum_{i=1}^{m} \boldsymbol{f}_i(\boldsymbol{x}_i)$, where $\boldsymbol{x} \triangleq \text{col}\{ \boldsymbol{x}_1, \cdots, \boldsymbol{x}_m \} \in \mathbb{R}^{mn}$, and $\boldsymbol{f}:\mathbb{R}^{mn} \mapsto \mathbb{R}$.
$\boldsymbol{f}$ is assumed to be closed, proper, and convex in $\mathbb{R}^{mn}$.
\begin{remark}\label{remark:linear_con}
The convex set $\mathcal{X}_i$ may represent the intersection of multiple local convex sets for agent $i$.
Section~\ref{sec:app} introduces a distributed MPC problem that involves a linear equality constraint of the form $\bar{\boldsymbol{A}}_i \boldsymbol{x}_i = \bar{\boldsymbol{B}}_i$.
Solving the problem of interest \eqref{problem_interest} with or without these local linear equality constraints is essentially equivalent, as the nonempty intersection of a closed convex set and the hyperplane defined by the linear equations remains closed and convex. Additional discussion on this topic can be found in Section~\ref{sec:app}.
\end{remark}

\begin{remark}
In the context of distributed MPC, the decision variable $\boldsymbol{x}_i$ typically consists of the states and controls of agent $i$ over a prediction horizon.
The system dynamics for agent $i$ are forward-propagated as a set of linear equations involving $\boldsymbol{x}_i$.
Coordination requirements between agents can be formulated as edge agreements, while the convex set $\mathcal{X}_i$ could represent the state and control bounds.
Further details on these topics are provided in Section~\ref{sec:app}.
\end{remark}

\subsection{Assumptions}

This section introduces some assumptions and notations.
\begin{assumption} \label{assum:existence}
There exist at least one solution $\boldsymbol{x}^* \in \mathbb{R}^{mn}$, $\boldsymbol{x}^* \triangleq \emph{col}\{ \boldsymbol{x}^*_1, \cdots, \boldsymbol{x}^*_m \}$, to Problem \eqref{problem_interest} that satisfy the edge agreements \eqref{problem_interest:edge}.
\end{assumption}

Assumption \ref{assum:existence} guarantees that Problem \eqref{problem_interest} has solutions, which implies that $\mathcal{X}_i \neq \emptyset, \ \forall i \in \mathcal{V}$;
$\mathcal{X}_i \cap \mathcal{X}_j \neq \emptyset, \ \forall i, j \in \mathcal{V}$.
The following assumption is adopted to guarantee the consistency of edge agreements \eqref{problem_interest:edge}.

\begin{assumption} \label{assum:consistency} \emph{(\textbf{Consistency})}
Given the undirected graph $\mathbb{G} = \{ \mathcal{V}, \mathcal{E} \}$, the linear constraints for edge agreements \eqref{problem_interest:edge} are consistent with each other for each pair of neighboring agents $(i,j) \in \mathcal{E}$, that is, $\boldsymbol{A}_{ij} = \boldsymbol{A}_{ji}, \  \boldsymbol{b}_{ij} = - \boldsymbol{b}_{ji}$.
\end{assumption}
The edge agreements \eqref{problem_interest:edge} are equivalent to some constraints using projection matrices. Let $\boldsymbol{P}_{ij} \in \mathbb{R}^{n \times n}$ denote a projection matrix and $\bar{\boldsymbol{b}}_{ij} \in \mathbb{R}^n$ denote as follows
\begin{equation} \label{eq:mat_define}
\boldsymbol{P}_{ij} = \boldsymbol{A}_{ij}^{\top} (\boldsymbol{A}_{ij} \boldsymbol{A}_{ij}^{\top})^{-1} \boldsymbol{A}_{ij}, \  \bar{\boldsymbol{b}}_{ij} = \boldsymbol{A}_{ij}^{\top} (\boldsymbol{A}_{ij} \boldsymbol{A}_{ij}^{\top})^{-1} \boldsymbol{b}_{ij}.
\end{equation}
Then denote the followings:
\begin{subequations}
\begin{align}
\bar{\boldsymbol{P}} &= \diag \{ \boldsymbol{P}_{i_1, j_1}, \boldsymbol{P}_{i_2, j_2}, \cdots, \boldsymbol{P}_{i_{\bar{m}},  j_{\bar{m}}} \} \in \mathbb{R}^{\bar{m}n \times \bar{m}n}, \label{eq:diag_mat_define} \\
\bar{\boldsymbol{b}} &= \text{col} \{ \bar{\boldsymbol{b}}_{i_1, j_1}, \cdots, \bar{\boldsymbol{b}}_{i_l, j_l}, \cdots, \bar{\boldsymbol{b}}_{i_{\bar{m}},  j_{\bar{m}}}  \}  \in \mathbb{R}^{\bar{m}n}, \label{eq:edge_b_define}
\end{align}
\end{subequations}
where $(i_l, j_l)$ denotes the $l$-th edge of the graph $\mathbb{G}$. By \eqref{eq:mat_define} and \eqref{eq:diag_mat_define}, note that $\bar{\boldsymbol{P}}^2 = \bar{\boldsymbol{P}}$ and $\bar{\boldsymbol{P}}^{\top} = \bar{\boldsymbol{P}}$. Based on the definition of $\boldsymbol{P}_{ij}$ and $\bar{\boldsymbol{b}}_{ij}$ in \eqref{eq:mat_define} and the fact that $\boldsymbol{P}_{ij} \bar{\boldsymbol{b}}_{ij} = \boldsymbol{b}_{ij}$, the linear constraints \eqref{problem_interest:edge} for edge agreements are equivalent to
\begin{equation} \label{eq:edge_agreement_each_after}
\boldsymbol{P}_{ij} (\boldsymbol{x}_i-\boldsymbol{x}_j-\bar{\boldsymbol{b}}_{ij}) = \boldsymbol{0}, \ \forall (i,j) \in \mathcal{E}.
\end{equation}

For the $m$-node-$\bar{m}$-edge undirected graph $\mathbb{G}$, one defines the oriented incidence matrix of $\mathbb{G}$ denoted by $\boldsymbol{H} \in \mathbb{R}^{\bar{m} \times m}$ such that its entry at the $k$-th row and the $j$-th column is 1 if edge $k$ is an incoming edge to node $j$; -1 if edge $k$ is an outgoing edge from node $j$; and 0 elsewhere.
Note that for undirected graphs, the direction for each edge could be arbitrary, as long as it is consistent with each $\boldsymbol{b}_{ij}$, i.e., satisfying Assumption \ref{assum:consistency}.
Based on \eqref{eq:edge_agreement_each_after}, the definitions of $H$ and $\bar{\boldsymbol{P}}$, and
\begin{equation} \label{eq:H_bar_define}
\mathbb{R}^{\bar{m}n \times mn} \ni \bar{\boldsymbol{H}} \triangleq \boldsymbol{H} \otimes \boldsymbol{I}_n ,
\end{equation} Let $\boldsymbol{x} \triangleq \{ \boldsymbol{x}_1, \cdots, \boldsymbol{x}_m  \} \in \mathbb{R}^{mn}$. Then by the definition of $\bar{\boldsymbol{H}}$ and $\bar{\boldsymbol{P}}$,
the following lemma holds:
\begin{lemma} \cite[Lemma~1]{lu2024distributed} \label{lemma:edge_agreement}
The edge agreements \eqref{problem_interest:edge} are equivalent to the following equation:
\begin{equation} \label{eq:edge_agreement_each_after1}
\bar{\boldsymbol{P}}(\bar{\boldsymbol{H}} \boldsymbol{x}- \bar{\boldsymbol{b}}) = \boldsymbol{0},
\end{equation}
where $\bar{\boldsymbol{P}}$, $\bar{\boldsymbol{b}}$, and $\bar{\boldsymbol{H}}$ are as defined in \eqref{eq:diag_mat_define}, \eqref{eq:edge_b_define}, and \eqref{eq:H_bar_define}, respectively.
\end{lemma}

The following assumption about $\bar{\boldsymbol{H}}$ and $\bar{\boldsymbol{P}}$ is adopted.
\begin{assumption} \label{assump:well_configured}
The graph $\mathbb{G}$ is connected and well-configured for edge agreements, i.e.
\begin{equation} \label{eq:well_configured}
\emph{ker } \bar{\boldsymbol{H}}^{\top} \cap \emph{image } \bar{\boldsymbol{P}} = \boldsymbol{0}.
\end{equation}
\end{assumption}

By the definition of kernel and image, $\text{ker } \bar{\boldsymbol{H}}^{\top} \triangleq \{ \bar{\boldsymbol{x}} \in \mathbb{R}^{\bar{m}n} \  \vert  \ \bar{\boldsymbol{H}}^{\top} \bar{\boldsymbol{x}} = \boldsymbol{0} \}$ and $\text{image } \bar{\boldsymbol{P}} \triangleq \{ \boldsymbol{y} \in \mathbb{R}^{\bar{m}n} \  \vert  \  \boldsymbol{y} = \bar{\boldsymbol{P}} \bar{\boldsymbol{x}}, \ \bar{\boldsymbol{x}} \in \mathbb{R}^{\bar{m}n} \}$. Since $\bar{\boldsymbol{P}}$ is a diagonal matrix of projections, $\text{image } \bar{\boldsymbol{P}} = \mathbb{R}^{\bar{m}n}$. Then $\emph{ker } \bar{\boldsymbol{H}}^{\top} \cap \emph{image } \bar{\boldsymbol{P}} = \boldsymbol{0}$ indicates that
$\bar{\boldsymbol{H}}^{\top} \bar{\boldsymbol{P}} \bar{\boldsymbol{x}} = \boldsymbol{0} \Rightarrow \bar{\boldsymbol{P}} \bar{\boldsymbol{x}} = \boldsymbol{0}$, which further implies that 
$\bar{\boldsymbol{H}}^{\top}\bar{\boldsymbol{P}}(\bar{\boldsymbol{H}} \boldsymbol{x} - \bar{\boldsymbol{b}}) = \boldsymbol{0} \Rightarrow \bar{\boldsymbol{P}}(\bar{\boldsymbol{H}} \boldsymbol{x} - \bar{\boldsymbol{b}}) = \boldsymbol{0}$.

Given an arbitrary closed and convex set $\mathcal{X} \subset \mathbb{R}^n$, define an indicator function of a convex set, $\mathcal{I}_{\mathcal{X}}(\cdot): \mathbb{R}^n \mapsto \{0\}\cup\{\infty\} \subset \mathbb{R}\cup\{\infty\}$, as follows:
\begin{equation} \label{eq;cvx_set_indicator_define}
\mathcal{I}_{\mathcal{X}}(\boldsymbol{x}) = 
\begin{cases}
0, & \boldsymbol{x} \in \mathcal{X} \\
\infty, & \boldsymbol{x} \notin \mathcal{X}
\end{cases}
\end{equation}
$\mathcal{I}_{\mathcal{X}}(\boldsymbol{x})$ can be proved, by definition, closed, proper, and (not strictly) convex.

\section{Algorithm and Analysis} \label{sec:main_result}

This section presents theoretical results for the distributed optimization problem under edge agreements.
First, a discrete-time distributed alternating direction method of multipliers (ADMM) is proposed to solve the problem of interest, where the decision variables are constrained by local convex sets.
Second, a main theorem is provided to establish the convergence of the proposed algorithm, supported by a theoretical analysis.
Third, a numerical simulation is performed to validate the proposed algorithm.

\subsection{Proposed Distributed Algorithm} \label{subsec:proposed_algo}

First, the following lemma reformulates the problem \eqref{problem_interest}.
\begin{lemma} \label{lemma:convert_problem}
Let Assumption \ref{assum:existence}, \ref{assum:consistency}, and \ref{assump:well_configured} hold. The problem of interest \eqref{problem_interest} is equivalent to the following optimization:
\begin{mini!}|s|
{{\{\boldsymbol{x}_i,\boldsymbol{z}_i\}}_{i=1}^m}{ \boldsymbol{\ell}(\boldsymbol{x},\boldsymbol{z}) := \sum_{i=1}^{m} \boldsymbol{f}_i(\boldsymbol{x}_i) + \sum_{i=1}^{m} \mathcal{I}_{\mathcal{X}_i }(\boldsymbol{z}_i) \label{problem_interest_revised_2:obj}}
{\label{problem_interest_revised_2}}{}
\addConstraint{ \boldsymbol{x} = \boldsymbol{z} \label{problem_interest_revised_2:equality}}
\addConstraint{ \bar{\boldsymbol{H}}^{\top}\bar{\boldsymbol{P}}(\bar{\boldsymbol{H}} \boldsymbol{x} - \bar{\boldsymbol{b}}) = \boldsymbol{0}_{mn} \label{problem_interest_revised_2:edge}}
\end{mini!}
where $\boldsymbol{z}_i \in \mathbb{R}^n$ is a new decision variable for agent $i$ and $\boldsymbol{z} \triangleq \emph{col} \{ \boldsymbol{z}_1, \cdots, \boldsymbol{z}_m \} \in \mathbb{R}^{mn}$.
\end{lemma}
\begin{proof}
By Lemma \ref{lemma:edge_agreement} and Assumption \ref{assump:well_configured}, $\bar{\boldsymbol{P}}(\bar{\boldsymbol{H}} \boldsymbol{x} - \bar{\boldsymbol{b}}) = \boldsymbol{0}_{\bar{m}n} \Leftrightarrow \bar{\boldsymbol{H}}^{\top} \bar{\boldsymbol{P}}(\bar{\boldsymbol{H}} \boldsymbol{x} - \bar{\boldsymbol{b}}) = \boldsymbol{0}_{mn}$, thus  \eqref{problem_interest} and \eqref{problem_interest_revised_2} are equivalent to each other.
Together with the indicator function \eqref{eq;cvx_set_indicator_define} and the constraint $\boldsymbol{x}=\boldsymbol{z}$, \eqref{problem_interest_revised_2} is equivalent to \eqref{problem_interest}.
\end{proof}
Note that $\bar{\boldsymbol{P}}(\bar{\boldsymbol{H}} \boldsymbol{x} - \bar{\boldsymbol{b}}) = \boldsymbol{0}_{\bar{m}n}$ is the edge-wise constraint, where as $\bar{\boldsymbol{H}}^{\top} \bar{\boldsymbol{P}}(\bar{\boldsymbol{H}} \boldsymbol{x} - \bar{\boldsymbol{b}}) = \boldsymbol{0}_{mn}$ is the agent-wise constraint. These two are equivalent to each other if Assumption~\ref{assump:well_configured} hold.

By introducing the new variable $\boldsymbol{z}_i$ and the projection $\mathcal{I}_{\mathcal{X}_i}$, the convex set constraint becomes local to each $\boldsymbol{z}_i$. Thus, both the updates on $\boldsymbol{x}$ and $\boldsymbol{z}$ can be distributed locally for each agent $i$.
Without loss of generality (WLOG), the constraints in \eqref{problem_interest_revised_2} can be written as the following compact form:
\begin{equation} \label{eq:constraints_compact}
\boldsymbol{C} \boldsymbol{x} + \boldsymbol{D} \boldsymbol{z} = \boldsymbol{E},
\end{equation}
where $\boldsymbol{C} \triangleq \matt{\boldsymbol{I}_{mn} \\ \bar{\boldsymbol{H}}^{\top}\bar{\boldsymbol{P}}\bar{\boldsymbol{H}}} \in \mathbb{R}^{2mn \times mn}$, $\boldsymbol{D} \triangleq \matt{ -\boldsymbol{I}_{mn} \\ \boldsymbol{0}_{mn \times mn}} \in \mathbb{R}^{2mn \times mn}$, and $\boldsymbol{E} \triangleq \matt{\boldsymbol{0}_{mn} \\ \bar{\boldsymbol{H}}^{\top}\bar{\boldsymbol{P}}\bar{\boldsymbol{b}}} \in \mathbb{R}^{2mn}$.
Denote the constraint residual at iteration $k$ as
\begin{equation} \label{eq:constraint_residual}
\boldsymbol{r}_k \triangleq \boldsymbol{C} \boldsymbol{x}_k + \boldsymbol{D} \boldsymbol{z}_k - \boldsymbol{E} \in \mathbb{R}^{2mn}.
\end{equation}
This paper only uses the notation of \eqref{eq:constraints_compact} to simplify the notations in proofs.

The augmented Lagrangian $L_{\rho}(\boldsymbol{x}, \boldsymbol{z}, \boldsymbol{y})$ of Problem \eqref{problem_interest_revised_2} is 
\begin{equation} \label{eq:augmented_lagrangian_z}
\begin{split}
&L_{\rho}(\boldsymbol{x}, \boldsymbol{z}, \boldsymbol{y}) \triangleq \sum_{i=1}^{m} \boldsymbol{f}_i(\boldsymbol{x}_i) + \sum_{i=1}^{m} \mathcal{I}_{\mathcal{X}_i}(\boldsymbol{z}_i) + \boldsymbol{\lambda}^{\top}(\boldsymbol{x}-\boldsymbol{z}) + \\
& \frac{\rho}{2} \norm{\boldsymbol{x}-\boldsymbol{z}}^2 + \boldsymbol{\mu}^{\top} \bar{\boldsymbol{H}}^{\top}\bar{\boldsymbol{P}}(\bar{\boldsymbol{H}} \boldsymbol{x} - \bar{\boldsymbol{b}}) + \frac{\rho}{2} \norm{\bar{\boldsymbol{P}}(\bar{\boldsymbol{H}} \boldsymbol{x} - \bar{\boldsymbol{b}})}^2,
\end{split}
\end{equation}
where $\rho > 0$ is the penalty parameter, $\boldsymbol{y} \triangleq \text{col} \{ \boldsymbol{\lambda}, \boldsymbol{\mu} \} \in \mathbb{R}^{2mn}$ ($\boldsymbol{\lambda}, \boldsymbol{\mu} \in \mathbb{R}^{mn}$) is the Lagrangian multiplier associated with the constraints \eqref{problem_interest_revised_2:equality} and \eqref{problem_interest_revised_2:edge}.
Note that $\bar{\boldsymbol{H}}\boldsymbol{\mu} \in \mathbb{R}^{\bar{m}n}$ is the edge-wise multiplier associated with $\bar{\boldsymbol{P}}(\bar{\boldsymbol{H}} \boldsymbol{x} - \bar{\boldsymbol{b}}) \in \mathbb{R}^{\bar{m}n}$, whereas $\boldsymbol{\mu} \in \mathbb{R}^{mn}$ is the agent-wise multiplier associated with $\bar{\boldsymbol{H}}^{\top}\bar{\boldsymbol{P}}(\bar{\boldsymbol{H}} \boldsymbol{x} - \bar{\boldsymbol{b}}) \in \mathbb{R}^{mn}$.
With Assumption~\ref{assump:well_configured}, these two constraints are equivalent to each other.
A common assumption on saddle points is adopted as follows.
\begin{assumption} \label{assump:saddle_2}
$\boldsymbol{f}=\sum_{i=1}^m \boldsymbol{f}_i$ is closed, proper, and convex. The unaugmented Lagrangian $L_0(\boldsymbol{x}, \boldsymbol{z}, \boldsymbol{y})$ has a saddle point.
\end{assumption}
By Assumption \ref{assump:saddle_2}, there exists $\boldsymbol{x}^*$, $\boldsymbol{z}^*$ and $\boldsymbol{y}^*$, not necessarily unique, where $L_0(\boldsymbol{x}^*, \boldsymbol{z}^*, \boldsymbol{y}) \leq L_0(\boldsymbol{x}^*, \boldsymbol{z}^*, \boldsymbol{y}^*) \leq L_0(\boldsymbol{x}, \boldsymbol{z}, \boldsymbol{y}^*)$ holds for all $\boldsymbol{x}$, $\boldsymbol{z}$ and $\boldsymbol{y}$.

Given the augmented Lagrangian $L_{\rho}(\boldsymbol{x}, \boldsymbol{z}, \boldsymbol{y})$ in \eqref{eq:augmented_lagrangian_z}, a distributed algorithm based on the alternating direction method of multipliers (ADMM) is proposed as follows:
\begin{subequations} \label{eq:distributed_update_2}
\begin{align}
\begin{split} \label{eq:distributed_update_2:x}
&\boldsymbol{x}_{i,k+1} = \arg\min_{\boldsymbol{x}_i} \{ \boldsymbol{f}_i(\boldsymbol{x}_i) + \boldsymbol{\lambda}_{i,k}^{\top}\boldsymbol{x}_i + \frac{\rho}{2} \norm{\boldsymbol{x}_i-\boldsymbol{z}_{i,k}}^2 \\
&+\boldsymbol{\mu}_{i,k}^{\top}\textstyle\sum_{j \in \mathcal{N}_i} \boldsymbol{P}_{ij}(\boldsymbol{x}_i - \boldsymbol{x}_{j,k} - \bar{\boldsymbol{b}}_{ij}) \\
&+\frac{\rho}{2}\textstyle\sum_{j \in \mathcal{N}_i} \norm{\boldsymbol{P}_{ij}(\boldsymbol{x}_i - \boldsymbol{x}_{j,k} - \bar{\boldsymbol{b}}_{ij})}^2 \},
\end{split} \\
& \boldsymbol{z}_{i,k+1} = \arg\min_{\boldsymbol{z}_i \in \mathcal{X}_i} \norm{\boldsymbol{z}_i - (\boldsymbol{x}_{i,k+1}+\boldsymbol{\lambda}_{i,k}/\rho)}^2, \label{eq:distributed_update_2:z}\\
& \boldsymbol{\lambda}_{i,k+1} = \boldsymbol{\lambda}_{i,k} + \boldsymbol{x}_{i,k+1} - \boldsymbol{z}_{i,k+1}, \label{eq:distributed_update_2:lambda}\\
& \boldsymbol{\mu}_{i,k+1} = \boldsymbol{\mu}_{i,k} + \textstyle\sum_{j \in \mathcal{N}_i} \boldsymbol{P}_{ij}(\boldsymbol{x}_{i,k+1} - \boldsymbol{x}_{j,k+1} - \bar{\boldsymbol{b}}_{ij}), \label{eq:distributed_update_2:mu}
\end{align}
\end{subequations}
where $\boldsymbol{\lambda}_{i,k} \in \mathbb{R}^{n}$ and $\boldsymbol{\mu}_{i,k} \in \mathbb{R}^{n}$ are the Lagrangian multipliers associated with agent $i$ at iteration $k$, and $\boldsymbol{\lambda}_k \coloneq \text{col}\{ \boldsymbol{\lambda}_{1,k}, \cdots, \boldsymbol{\lambda}_{m,k} \}$, $\boldsymbol{\mu}_k \coloneq \text{col}\{ \boldsymbol{\mu}_{1,k}, \cdots, \boldsymbol{\mu}_{m,k} \}$.
The following theorem guarantees the convergence of the proposed rule \eqref{eq:distributed_update_2}. And its theoretical analysis is provided in Section~\ref{subsec:analysis}.
The proposed algorithm is summarized in Algorithm~\ref{alg:algo_proposed}, where the content within \textbf{parfor} is executed by each agent parallelly.


\begin{theorem} \label{theorem:update_2}
Let Assumption \ref{assum:existence}, \ref{assum:consistency}, \ref{assump:well_configured}, and \ref{assump:saddle_2} hold.
Define the optimal objective $\boldsymbol{\ell}^*$ as
\begin{equation} \label{eq:optimal_obj_define}
\boldsymbol{\ell}^* \triangleq \inf \{  \boldsymbol{\ell}(\boldsymbol{x},\boldsymbol{z}) \  \vert  \ \boldsymbol{z}_i \in \mathcal{X}_i, \boldsymbol{x}=\boldsymbol{z},\eqref{problem_interest:edge} \}.
\end{equation}
For all finite initial $(\boldsymbol{x}_{i,0}, \boldsymbol{z}_{i,0}, \boldsymbol{y}_{i,0})$, the proposed distributed update rule \eqref{eq:distributed_update_2} drives each $\boldsymbol{x}_{i,k}$, $\boldsymbol{z}_{i,k}$, and $\boldsymbol{y}_{i,k}$ such that:
\begin{enumerate}
\item $\boldsymbol{x}_{i,k} - \boldsymbol{z}_{i,k} \to \boldsymbol{0}$ as $k \to \infty$, $\forall i \in \mathcal{V}$;
\item $\boldsymbol{A}_{ij}(\boldsymbol{x}_{i,k} - \boldsymbol{x}_{j,k})-\boldsymbol{b}_{ij} \to \boldsymbol{0}$ as $k \to \infty$, $\forall (i,j) \in \mathcal{E}$;
\item $\boldsymbol{\ell}(\boldsymbol{x}_k,\boldsymbol{z}_k) \to \boldsymbol{\ell}^*$ as $k \to \infty$;
\item $\boldsymbol{y}_k$ converges to an optimal solution $\boldsymbol{y}^*$ of the dual problem; $(\boldsymbol{x}_k,\boldsymbol{z}_k)$ converges to an optimal solution $(\boldsymbol{x}^*,\boldsymbol{z}^*)$ of Problem \eqref{problem_interest_revised_2}, where $\boldsymbol{x}^*$ is an optimal solution of Problem \eqref{problem_interest}.
\end{enumerate}
\end{theorem}

\begin{algorithm}
\caption{Proposed Distributed Algorithm}\label{alg:algo_proposed}
\DontPrintSemicolon
\KwIn{$\rho$, $N_{\mathrm{iter}} \in \mathbb{Z}_+$}
$k \gets 0$; Initialize $\boldsymbol{x}_{i,0}, \boldsymbol{z}_{i,0}, \boldsymbol{\lambda}_{i,0}, \boldsymbol{\mu}_{i,0}, \forall i$\;
Each agent $i$ obtains $\boldsymbol{x}_{j,0}$ from neighbors $j \in \mathcal{N}_i$\;
\While {$k < N_{\mathrm{iter}}$} {
$\textbf{par}$\For {\emph{agent} $i = 1 \ to \ m$} {
$\boldsymbol{x}_i(k+1) \gets$ \eqref{eq:distributed_update_2:x}\;
$\boldsymbol{z}_i(k+1) \gets$ \eqref{eq:distributed_update_2:z}\;
$\boldsymbol{\lambda}_i(k+1) \gets$ \eqref{eq:distributed_update_2:lambda}\;
Obtain $\boldsymbol{x}_{j,k+1}$ from neighbors $j \in \mathcal{N}_i$\;
$\boldsymbol{\mu}_i(k+1) \gets$ \eqref{eq:distributed_update_2:mu}\;
}
$k \gets k+1$\;
stop when a prescribed stopping criterion is met\;
}
$\textbf{Return}$ $\{\boldsymbol{z}_{i,k+1}\}_{i=1}^m$ as solution to Problem \eqref{problem_interest}\;
\end{algorithm}

\subsection{Theoretical Analysis} \label{subsec:analysis}

This subsection provides a theoretical analysis to prove Theorem~\ref{theorem:update_2}.
Given Assumption \ref{assum:existence}, \ref{assum:consistency}, and \ref{assump:well_configured} and Lemma~\ref{lemma:convert_problem}, Theorem~\ref{theorem:update_2} is equivalent to proving the following statements:
\begin{enumerate}[label=(\roman*)]
\item Residual convergence: $\boldsymbol{r}_k \to \boldsymbol{0}$ as $k \to \infty$;
\item Objective convergence: $\boldsymbol{\ell}(\boldsymbol{x}_k, \boldsymbol{z}_k) \to \boldsymbol{\ell}^*$ as $k \to \infty$;
\item Primal convergence: $\boldsymbol{x}_k \to \boldsymbol{x}^*$ and $\boldsymbol{z}_k \to \boldsymbol{z}^*$, as $k \to \infty$;
\item Dual convergence: $\boldsymbol{y}_k \to \boldsymbol{y}^*$ as $k \to \infty$.
\end{enumerate}
This subsection provides a theoretical analysis for both the aforementioned statements and Theorem~\ref{theorem:update_2}. Here, an analysis outline is presented to facilitate reading.
First, the equivalent centralized compact form of the proposed distributed rule \eqref{eq:distributed_update_2} is formulated as \eqref{eq:admm_original}.
Second, Lemma~\ref{lemma:inequality_1}, \ref{lemma:inequality_2}, and \ref{lemma:inequality_3} are presented as necessary to prove statements (i) and (ii).
Third, one can prove statement (iii).
Finally, with Assumption~\ref{assump:saddle_2}, one can prove statement (iv). The overall analysis is summarized in the Proof of Theorem~\ref{theorem:update_2}.

Given the augmented Lagrangian $L_{\rho}(\boldsymbol{x}, \boldsymbol{z}, \boldsymbol{y})$ of Problem \eqref{problem_interest_revised_2}, the alternating direction method of multipliers (ADMM) consists of the following centralized iterations:
\begin{subequations} \label{eq:admm_original}
\begin{align}
\boldsymbol{x}_{k+1} &= \arg\min_{\boldsymbol{x}} L_{\rho}(\boldsymbol{x}, \boldsymbol{z}_k, \boldsymbol{y}_k), \label{eq:admm_original:x} \\
\boldsymbol{z}_{k+1} &= \arg\min_{\{\boldsymbol{z}_i\in \mathcal{X}_i\}} L_{\rho}(\boldsymbol{x}_{k+1}, \boldsymbol{z}, \boldsymbol{y}_k), \label{eq:admm_original:z} \\
\boldsymbol{y}_{k+1} &= \boldsymbol{y}_k + \rho (\boldsymbol{C} \boldsymbol{x}_{k+1} + \boldsymbol{D} \boldsymbol{z}_{k+1} - \boldsymbol{E}). \label{eq:admm_original:multiplier}
\end{align}
\end{subequations}
$L_{\rho}(\boldsymbol{x}, \boldsymbol{z}_k, \boldsymbol{y}_k)$ in the update rule \eqref{eq:admm_original:x} can be simplified as
\begin{equation*}
\begin{split}
&L_{\rho}(\boldsymbol{x}, \boldsymbol{z}_k, \boldsymbol{y}_k) = \boldsymbol{f}(\boldsymbol{x}_k) + \cancelto{0}{{\sum_{i=1}^{m} \mathcal{I}_{\mathcal{X}_i}(\boldsymbol{z}_{i,k})} }+ \boldsymbol{\lambda}_k^\top(\boldsymbol{x}-\boldsymbol{z}_k) + \\
& \frac{\rho}{2} \norm{\boldsymbol{x}-\boldsymbol{z}_{k}}^2 + \boldsymbol{\mu}_k^\top(\bar{\boldsymbol{H}}^{\top}\bar{\boldsymbol{P}}(\bar{\boldsymbol{H}} \boldsymbol{x} - \bar{\boldsymbol{b}})) + \frac{\rho}{2}  \norm{\bar{\boldsymbol{P}}(\bar{\boldsymbol{H}} \boldsymbol{x} - \bar{\boldsymbol{b}})}^2.
\end{split}
\end{equation*}
$L_{\rho}(\boldsymbol{x}_{k+1}, \boldsymbol{z}, \boldsymbol{y}_k)$ in the update rule \eqref{eq:admm_original:z} can be written in the following quadratic form:
\begin{equation*}
\begin{split}
&L_{\rho}(\boldsymbol{x}_{k+1}, \boldsymbol{z}, \boldsymbol{y}_k) = \frac{\rho}{2} \norm{\boldsymbol{x}_{k+1}-\boldsymbol{z}}^2 + \boldsymbol{\lambda}_k^\top(\boldsymbol{x}_{k+1}-\boldsymbol{z}) = \\
& \frac{\rho}{2} \norm{\boldsymbol{z}-(\boldsymbol{x}_{k+1}+\frac{1}{\rho}\boldsymbol{\lambda}_{k})}^2 - \boldsymbol{\lambda}_k^\top (\frac{1}{2\rho}\boldsymbol{\lambda}_{k}+\boldsymbol{x}_{k+1}) + \boldsymbol{\lambda}_k^\top \boldsymbol{x}_{k+1}.
\end{split}
\end{equation*}
Thus, $\arg\min_{\{\boldsymbol{z}_i\in \mathcal{X}_i\}} L_{\rho}(\boldsymbol{x}_{k+1}, \boldsymbol{z}, \boldsymbol{y}_k)$ can be simplified as
\begin{equation*}
\arg\min_{\{\boldsymbol{z}_i\in \mathcal{X}_i\}} \norm{\boldsymbol{z}-(\boldsymbol{x}_{k+1}+\boldsymbol{\lambda}_{k}/\rho)}^2.
\end{equation*}
The update rule \eqref{eq:admm_original:multiplier} can be expanded as follows:
\begin{subequations} \label{eq:admm_2}
\begin{align}
\boldsymbol{\lambda}_{k+1} &= \boldsymbol{\lambda}_k + \boldsymbol{x}_{k+1} - \boldsymbol{z}_{k+1}, \label{eq:admm_2:lambda} \\
\boldsymbol{\mu}_{k+1} &= \boldsymbol{\mu}_k + \bar{\boldsymbol{H}}^{\top}\bar{\boldsymbol{P}}(\bar{\boldsymbol{H}} \boldsymbol{x}_{k+1} - \bar{\boldsymbol{b}}).\label{eq:admm_2:mu}
\end{align}
\end{subequations}


\begin{lemma} \label{lemma:inequality_1}
Let Assumption \ref{assum:existence}, \ref{assum:consistency}, \ref{assump:well_configured}, and \ref{assump:saddle_2} hold. Given $\boldsymbol{\ell}^*$ defined in \eqref{eq:optimal_obj_define}, the following inequality holds:
\begin{equation} \label{eq:lemma_inequality_1}
\boldsymbol{\ell}^* - \boldsymbol{\ell}_{k+1} \leq {\boldsymbol{y}^*}^{\top} \boldsymbol{r}_{k+1}.
\end{equation}
\end{lemma}
\begin{proof}
According to Assumption \ref{assump:saddle_2}, $(\boldsymbol{x}^*, \boldsymbol{z}^*, \boldsymbol{y}^*)$ is a saddle point for $L_0$, thus
\begin{equation} \label{eq:lemma_ineq_L}
L_0(\boldsymbol{x}^*, \boldsymbol{z}^*, \boldsymbol{y}^*) \leq L_0(\boldsymbol{x}_{k+1}, \boldsymbol{z}_{k+1}, \boldsymbol{y}^*).
\end{equation}
Given the definition of $L_0(\cdot)$ in \eqref{eq:augmented_lagrangian_z} and $\boldsymbol{C} \boldsymbol{x}^* + \boldsymbol{D} \boldsymbol{z}^* = \boldsymbol{E}$, expanding $L_0(\boldsymbol{x}^*, \boldsymbol{z}^*, \boldsymbol{y}^*)$  yields
\begin{equation*}
L_0(\boldsymbol{x}^*, \boldsymbol{z}^*, \boldsymbol{y}^*) = \boldsymbol{\ell}^*.
\end{equation*}
With $\boldsymbol{\ell}_{k+1} = \boldsymbol{f}(\boldsymbol{x}_{k+1}) + \sum_{i=1}^m \mathcal{I}_{\mathcal{X}_i}(\boldsymbol{z}_{i,k+1})$, \eqref{eq:lemma_ineq_L} reads
\begin{equation*}
\boldsymbol{\ell}^* \leq \boldsymbol{\ell}_{k+1} + {\boldsymbol{y}^*}^\top \boldsymbol{r}_{k+1} \Rightarrow \boldsymbol{\ell}^* - \boldsymbol{\ell}_{k+1} \leq {\boldsymbol{y}^*}^{\top} \boldsymbol{r}_{k+1}.
\end{equation*}
This completes the proof.
\end{proof}

\begin{lemma} \label{lemma:inequality_2}
Let Assumption \ref{assum:existence}, \ref{assum:consistency}, \ref{assump:well_configured}, and \ref{assump:saddle_2} hold. The following inequality holds:
\begin{equation} \label{eq:lemma_inequality_2}
\begin{split}
\boldsymbol{\ell}_{k+1} - \boldsymbol{\ell}^* \leq &-\boldsymbol{y}_{k+1}^\top \boldsymbol{r}_{k+1} - \rho (\boldsymbol{D}(\boldsymbol{z}_{k+1} - \boldsymbol{z}_k))^\top \cdot \\
&(-\boldsymbol{r}_{k+1} + \boldsymbol{D}(\boldsymbol{z}_{k+1} - \boldsymbol{z}^*)).
\end{split}
\end{equation}
\end{lemma}
\begin{proof}
By \eqref{eq:admm_original:x}, $\boldsymbol{x}_{k+1}$ minimizes $L_{\rho}(\boldsymbol{x}, \boldsymbol{z}_k, \boldsymbol{y}_k)$.
By Assumption \ref{assump:saddle_2}, $\boldsymbol{f}$ is closed, proper, and convex, thereby sub-differentiable.
Thus $L_{\rho}(\boldsymbol{x}, \boldsymbol{z}_k, \boldsymbol{y}_k)$ is subdifferentiable to $\boldsymbol{x}_{k+1}$.
The necessary and sufficient optimality condition is
\begin{equation} \label{eq:optimality_lemma}
\begin{split}
\boldsymbol{0} \in &\partial L_{\rho}(\boldsymbol{x}_{k+1}, \boldsymbol{z}_k, \boldsymbol{y}_k) = \partial \boldsymbol{f}(\boldsymbol{x}_{k+1}) + \boldsymbol{C}^\top \boldsymbol{y}_k + \\
& \rho \boldsymbol{C}^\top (\boldsymbol{C}\boldsymbol{x}_{k+1}+\boldsymbol{D} \boldsymbol{z}_{k} - \boldsymbol{E}).
\end{split}
\end{equation}
Since $\boldsymbol{y}_{k+1} = \boldsymbol{y}_k + \rho \boldsymbol{r}_{k+1}$, plugging $\boldsymbol{y}_k = \boldsymbol{y}_{k+1} - \rho \boldsymbol{r}_{k+1}$ in \eqref{eq:optimality_lemma} yields
\begin{equation} \label{eq:KKT_x}
\boldsymbol{0} \in \partial \boldsymbol{f}(\boldsymbol{x}_{k+1}) + \boldsymbol{C}^\top (\boldsymbol{y}_{k+1} - \rho \boldsymbol{D}(\boldsymbol{z}_{k+1} - \boldsymbol{z}_k)).
\end{equation}
This implies that $\boldsymbol{x}_{k+1}$ minimizes
\begin{equation} \label{eq:x_kp1_lemma}
\boldsymbol{f}(\boldsymbol{x}) + (\boldsymbol{y}_{k+1}-\rho \boldsymbol{D}(\boldsymbol{z}_{k+1}-\boldsymbol{z}_k))^\top \boldsymbol{C} \boldsymbol{x}.
\end{equation}
The matrix $\boldsymbol{C}^\top \boldsymbol{C} = \boldsymbol{I}_{mn} + (\bar{\boldsymbol{H}}^{\top}\bar{\boldsymbol{P}}\bar{\boldsymbol{H}})^\top (\bar{\boldsymbol{H}}^{\top}\bar{\boldsymbol{P}}\bar{\boldsymbol{H}})$ is positive definite, thereby invertible.
Thus, $L_{\rho}(\boldsymbol{x}, \boldsymbol{z}_k, \boldsymbol{y}_k)$ is the summation of a proper, closed, and convex function and a strictly convex quadratic function.
By \cite[Proposition~15.37]{gallier2019fundamentals}, the x-minimization step \eqref{eq:admm_original:x} has a unique solution.

Similarly, $\mathcal{I}_{\mathcal{X}_i}(\boldsymbol{z}_i)$ by definition is closed, proper, and convex, thereby subdifferentiable, where its subgradient is the normal cone of $\mathcal{X}_i$ at $\boldsymbol{z}_i$.
Thus, there holds
\begin{equation} \label{eq:KKT_z}
\boldsymbol{0} \in \partial (\sum_{i=1}^{m} \mathcal{I}_{\mathcal{X}_i}(\boldsymbol{z}_i)) + \boldsymbol{D}^\top \boldsymbol{y}_{k+1} = \sum_{i=1}^{m} \partial \mathcal{I}_{\mathcal{X}_i}(\boldsymbol{z}_i) + \boldsymbol{D}^\top \boldsymbol{y}_{k+1},
\end{equation}
which further implies that $\boldsymbol{z}_{k+1}$ minimizes $L_{\rho}(\boldsymbol{x}_{k+1}, \boldsymbol{z}, \boldsymbol{y}_{k})$ or equivalently $\sum_{i=1}^{m} \mathcal{I}_{\mathcal{X}_i}(\boldsymbol{z}_i) + \boldsymbol{y}_{k+1}^\top \boldsymbol{D} \boldsymbol{z}$.
The matrix $\boldsymbol{D}^\top \boldsymbol{D} = \boldsymbol{I}_{mn}$ is positive definite, thereby invertible.
Similarly, by \cite[Proposition~15.37]{gallier2019fundamentals}, the z-minimization step \eqref{eq:admm_original:z} has a unique solution.

Given \eqref{eq:x_kp1_lemma}, there holds
\begin{equation} \label{eq:x_kp1_lemma_ineq}
\begin{split}
&\boldsymbol{f}(\boldsymbol{x}_{k+1}) + (\boldsymbol{y}_{k+1}-\rho \boldsymbol{D}(\boldsymbol{z}_{k+1}-\boldsymbol{z}_k))^\top \boldsymbol{C} \boldsymbol{x}_{k+1} \leq \\
&\boldsymbol{f}(\boldsymbol{x}^*) + (\boldsymbol{y}_{k+1}-\rho \boldsymbol{D}(\boldsymbol{z}_{k+1}-\boldsymbol{z}_k))^\top \boldsymbol{C} \boldsymbol{x}^*.
\end{split}
\end{equation}
And given \eqref{eq:KKT_z}, there holds
\begin{equation} \label{eq:z_kp1_lemma_ineq}
\sum_{i=1}^{m} \mathcal{I}_{\mathcal{X}_i}(\boldsymbol{z}_{i,k+1}) + \boldsymbol{y}_{k+1}^\top \boldsymbol{D}\boldsymbol{z}_{k+1} \leq \sum_{i=1}^{m} \mathcal{I}_{\mathcal{X}_i}(\boldsymbol{z}_i^*) + \boldsymbol{y}_{k+1}^\top \boldsymbol{D}\boldsymbol{z}^*.
\end{equation}
With $\boldsymbol{C} \boldsymbol{x}^* + \boldsymbol{D} \boldsymbol{z}^* = \boldsymbol{E}$, adding \eqref{eq:x_kp1_lemma_ineq} and \eqref{eq:z_kp1_lemma_ineq} together yields
\begin{equation*}
\begin{split}
\boldsymbol{\ell}_{k+1} - \boldsymbol{\ell}^* \leq &-\boldsymbol{y}_{k+1}^\top \boldsymbol{r}_{k+1} - \rho (\boldsymbol{D}(\boldsymbol{z}_{k+1} - \boldsymbol{z}_k))^\top \cdot \\
&(-\boldsymbol{r}_{k+1} + \boldsymbol{D}(\boldsymbol{z}_{k+1} - \boldsymbol{z}^*)).
\end{split}
\end{equation*}
This completes the proof.
\end{proof}

Now define a Lyapunov function candidate as
\begin{equation}
V_k = \frac{1}{\rho} \norm{\boldsymbol{y}_k - \boldsymbol{y}^*}^2 + \rho \norm{\boldsymbol{D}(\boldsymbol{z}_k-\boldsymbol{z}^*)}^2.
\end{equation}

\begin{lemma} \label{lemma:inequality_3}
Let Assumption \ref{assum:existence}, \ref{assum:consistency}, \ref{assump:well_configured}, and \ref{assump:saddle_2} hold. The following inequality holds:
\begin{equation} \label{eq:lemma_inequality_3}
V_{k+1} \leq V_k - \rho \norm{\boldsymbol{r}_{k+1}}^2 - \rho \norm{\boldsymbol{D}(\boldsymbol{z}_{k+1} - \boldsymbol{z}_k)}^2.
\end{equation}
\end{lemma}
\begin{proof}
Since Assumption \ref{assump:saddle_2} holds, \eqref{eq:lemma_inequality_1} from Lemma \ref{lemma:inequality_1} and \eqref{eq:lemma_inequality_2} from Lemma \ref{lemma:inequality_2} hold.
Adding \eqref{eq:lemma_inequality_1} and \eqref{eq:lemma_inequality_2} together and multiplying by 2 yields
\begin{equation} \label{eq:lemma_ineq_add}
\begin{split}
&2(\boldsymbol{y}_{k+1}-\boldsymbol{y}^*)^\top \boldsymbol{r}_{k+1} - 2\rho(\boldsymbol{D}(\boldsymbol{z}_{k+1}-\boldsymbol{z}_k))^\top \boldsymbol{r}_{k+1} + \\
& 2\rho (\boldsymbol{D}(\boldsymbol{z}_{k+1}-\boldsymbol{z}_k))^\top (\boldsymbol{D}(\boldsymbol{z}_{k+1}-\boldsymbol{z}_k)) \leq \boldsymbol{0}.
\end{split}
\end{equation}

First, by substituting $\boldsymbol{y}_{k+1} = \boldsymbol{y}_k + \rho \boldsymbol{r}_{k+1}$, rewriting the first term of \eqref{eq:lemma_ineq_add} yields
\begin{equation} \label{eq:lemma_ineq_sub_1}
2(\boldsymbol{y}_k-\boldsymbol{y}^*)^\top \boldsymbol{r}_{k+1} + \rho \norm{\boldsymbol{r}_{k+1}}^2 + \rho \norm{\boldsymbol{r}_{k+1}}^2.
\end{equation}
Then substituting $\boldsymbol{r}_{k+1}=\frac{1}{\rho} (\boldsymbol{y}_{k+1}-\boldsymbol{y}_k)$ in the first two terms of \eqref{eq:lemma_ineq_sub_1} yields
\begin{equation} \label{eq:lemma_ineq_sub_2}
\frac{2}{\rho} (\boldsymbol{y}_k - \boldsymbol{y}^*)^\top (\boldsymbol{y}_{k+1} - \boldsymbol{y}_k) + \frac{1}{\rho} \norm{\boldsymbol{y}_{k+1}-\boldsymbol{y}_k}^2 + \rho \norm{\boldsymbol{r}_{k+1}}^2.
\end{equation}
Since $\boldsymbol{y}_{k+1} - \boldsymbol{y}_k = (\boldsymbol{y}_{k+1} - \boldsymbol{y}^*) - (\boldsymbol{y}_k - \boldsymbol{y}^*)$, \eqref{eq:lemma_ineq_sub_2} can be written as
\begin{equation} \label{eq:lemma_ineq_sub_3}
\frac{1}{\rho} (\norm{\boldsymbol{y}_{k+1}-\boldsymbol{y}^*}^2 - \norm{\boldsymbol{y}_k-\boldsymbol{y}^*}^2) + \rho \norm{\boldsymbol{r}_{k+1}}^2.
\end{equation}
Now the rest of the terms in \eqref{eq:lemma_ineq_add} and \eqref{eq:lemma_ineq_sub_1} is
$\rho \norm{\boldsymbol{r}_{k+1}}^2 - 2\rho(\boldsymbol{D}(\boldsymbol{z}_{k+1}-\boldsymbol{z}_k))^\top \boldsymbol{r}_{k+1} + 2\rho (\boldsymbol{D}(\boldsymbol{z}_{k+1}-\boldsymbol{z}_k))^\top (\boldsymbol{D}(\boldsymbol{z}_{k+1}-\boldsymbol{z}_k))$.
Substituting $\boldsymbol{z}_{k+1} - \boldsymbol{z}^* = (\boldsymbol{z}_{k+1} - \boldsymbol{z}_k) + (\boldsymbol{z}_k - \boldsymbol{z}^*)$
in the last term of the above expression and rearranging the expression into a quadratic form yields
\begin{equation}
\begin{split}
&\rho \norm{\boldsymbol{D}(\boldsymbol{z}_{k+1} - \boldsymbol{z}_k)}^2 + \rho \norm{\boldsymbol{r}_{k+1} - \boldsymbol{D}(\boldsymbol{z}_{k+1}-\boldsymbol{z}_k)}^2 \\
&+ 2\rho (\boldsymbol{D}(\boldsymbol{z}_{k+1}-\boldsymbol{z}_k))^\top(\boldsymbol{D}(\boldsymbol{z}_k-\boldsymbol{z}^*)).
\end{split}
\end{equation}
Substituting $\boldsymbol{z}_{k+1} - \boldsymbol{z}_k = (\boldsymbol{z}_{k+1}-\boldsymbol{z}^*) - (\boldsymbol{z}_k-\boldsymbol{z}^*)$ in the last two terms of the above expression yields
\begin{equation} \label{eq:lemma_ineq_sub_4}
\begin{split}
& \rho \norm{\boldsymbol{r}_{k+1}-\boldsymbol{D}(\boldsymbol{z}_{k+1}-\boldsymbol{z}_k)}^2 \\
& + \rho (\norm{\boldsymbol{D}(\boldsymbol{z}_{k+1}-\boldsymbol{z}^*)}^2 - \norm{\boldsymbol{D}(\boldsymbol{z}_k - \boldsymbol{z}^*)}^2).
\end{split}
\end{equation}

Thus, given \eqref{eq:lemma_ineq_sub_3} and \eqref{eq:lemma_ineq_sub_4}, \eqref{eq:lemma_ineq_add} can be written as
\begin{equation} \label{eq:lemma_ineq_sub_5}
V_k - V_{k+1} \geq \rho \norm{\boldsymbol{r}_{k+1} - \boldsymbol{D}(\boldsymbol{z}_{k+1} - \boldsymbol{z}_k)}^2.
\end{equation}
To prove the inequality \eqref{eq:lemma_inequality_3}, it is sufficient to show that the term $-2\rho \boldsymbol{r}_{k+1}^\top (\boldsymbol{D}(\boldsymbol{z}_{k+1}-\boldsymbol{z}_k))$ in the expanded form of the right-hand side of \eqref{eq:lemma_ineq_sub_5} is positive.

From \eqref{eq:KKT_z} of Lemma \ref{lemma:inequality_2}, $\boldsymbol{z}_{k+1}$ minimizes $\sum_{i=1}^{m} \mathcal{I}_{\mathcal{X}_i}(\boldsymbol{z}_i) + \boldsymbol{y}_{k+1}^\top \boldsymbol{D} \boldsymbol{z}$ and similarly, $\boldsymbol{z}_k$ minimizes $\sum_{i=1}^{m} \mathcal{I}_{\mathcal{X}_i}(\boldsymbol{z}_i) + \boldsymbol{y}_k^\top \boldsymbol{D} \boldsymbol{z}$.
Thus, there exist the following inequality expressions
\begin{equation*}
\begin{split}
&\sum_{i=1}^{m} \mathcal{I}_{\mathcal{X}_i}(\boldsymbol{z}_{i,k+1}) + \boldsymbol{y}_{k+1}^\top \boldsymbol{D} \boldsymbol{z}_{k+1} \leq \sum_{i=1}^{m} \mathcal{I}_{\mathcal{X}_i}(\boldsymbol{z}_{i,k}) + \boldsymbol{y}_{k+1}^\top \boldsymbol{D} \boldsymbol{z}_{k}, \\
&\sum_{i=1}^{m} \mathcal{I}_{\mathcal{X}_i}(\boldsymbol{z}_{i,k}) + \boldsymbol{y}_{k}^\top \boldsymbol{D} \boldsymbol{z}_{k} \leq \sum_{i=1}^{m} \mathcal{I}_{\mathcal{X}_i}(\boldsymbol{z}_{i,k+1}) + \boldsymbol{y}_{k}^\top \boldsymbol{D} \boldsymbol{z}_{k+1},
\end{split}
\end{equation*}
and adding the above two expressions together yields
\begin{equation}
(\boldsymbol{y}_{k+1} - \boldsymbol{y}_k)^\top (\boldsymbol{D}(\boldsymbol{z}_{k+1} - \boldsymbol{z}_k)) \leq 0.
\end{equation}
Given the update rule \eqref{eq:admm_original:multiplier}, the constraint residual $\boldsymbol{r}_{k+1}$ and $\rho > 0$, substituting $\boldsymbol{y}_{k+1} - \boldsymbol{y}_k = \rho \boldsymbol{r}_{k+1}$ yields $-2\rho \boldsymbol{r}_{k+1}^\top (\boldsymbol{D}(\boldsymbol{z}_{k+1}-\boldsymbol{z}_k)) \geq 0$.
Thus,
\begin{equation*}
\begin{split}
V_k - V_{k+1} &\geq \rho \norm{\boldsymbol{r}_{k+1} - \boldsymbol{D}(\boldsymbol{z}_{k+1} - \boldsymbol{z}_k)}^2 \\
&\geq \rho \norm{\boldsymbol{r}_{k+1}}^2 + \rho \norm{\boldsymbol{D}(\boldsymbol{z}_{k+1} - \boldsymbol{z}_k)}^2.
\end{split}
\end{equation*}
This completes the proof.
\end{proof}

Finally, the proof of Theorem \ref{theorem:update_2} is provided below.
\begin{proof}[\textbf{Proof of Theorem~\ref{theorem:update_2}}]

Let Assumption \ref{assum:existence}, \ref{assum:consistency}, \ref{assump:well_configured}, and \ref{assump:saddle_2} hold.
By Assumption~\ref{assump:saddle_2}, $ L_0(\boldsymbol{x}^*, \boldsymbol{z}^*, \boldsymbol{y}^*)$ is finite for any saddle point $(\boldsymbol{x}^*, \boldsymbol{z}^*, \boldsymbol{y}^*)$.
By \cite[Theorem~28.3]{Rockafellar+1970}, Assumption~\ref{assump:saddle_2} is equivalent to the fact that the KKT (Karush–Kuhn–Tucker) conditions are satisfied by the (not necessarily unique) saddle point $(\boldsymbol{x}^*, \boldsymbol{z}^*, \boldsymbol{y}^*)$,
\begin{equation} \label{eq:KKT_saddle}
\begin{split}
&\boldsymbol{C} \boldsymbol{x}^* + \boldsymbol{D} \boldsymbol{z}^* - \boldsymbol{E} = \boldsymbol{0}, \\
&\boldsymbol{0} \in \partial \boldsymbol{f}(\boldsymbol{x}^*) + \textstyle\sum_{i=1}^{m} \partial \mathcal{I}_{\mathcal{X}_i}(\boldsymbol{z}_i^*) + \boldsymbol{C}^\top \boldsymbol{y}^* + \boldsymbol{D}^\top \boldsymbol{y}^*.
\end{split}
\end{equation}
This also implies that $(\boldsymbol{x}^*, \boldsymbol{z}^*)$ is a solution to Problem \eqref{problem_interest_revised_2} and $\boldsymbol{y}^*$ is dual optimal and strong duality holds.

By Lemma~\ref{lemma:inequality_3}, there holds $V_{k+1} \leq V_k - \rho \norm{\boldsymbol{r}_{k+1}}^2 - \rho \norm{\boldsymbol{D}(\boldsymbol{z}_{k+1} - \boldsymbol{z}_k)}^2$.
So, there holds $V_k \leq V_0$, which further implies that $\boldsymbol{y}_k$ and $\boldsymbol{D}\boldsymbol{z}_k$ are bounded.
Iterating and adding the above inequality from 0 until $k$ yields
\begin{equation*}
V_{k+1} \leq V_0 - \rho \textstyle\sum_{k=0}^{\infty} (\norm{\boldsymbol{r}_{k+1}}^2 + \norm{\boldsymbol{D}(\boldsymbol{z}_{k+1}-\boldsymbol{z}_k)}^2).
\end{equation*}
Since $0 \leq V_{k+1} \leq V_0$, the above inequality implies that
\begin{equation*}
\rho \textstyle\sum_{k=0}^{\infty} (\norm{\boldsymbol{r}_{k+1}}^2 + \norm{\boldsymbol{D}(\boldsymbol{z}_{k+1}-\boldsymbol{z}_k)}^2) \leq V_0 - V_{k+1} \leq V_0,
\end{equation*}
which further implies that the series $\sum_{k=0}^{\infty} \boldsymbol{r}_k$ and $\sum_{k=0}^{\infty} \boldsymbol{D}(\boldsymbol{z}_{k+1} - \boldsymbol{z}_k)$ converge.
Thus, $\boldsymbol{r}_k \to \boldsymbol{0}$ and $\boldsymbol{D}(\boldsymbol{z}_{k+1}-\boldsymbol{z}_k) \to \boldsymbol{0}$ as $k \to \infty$.
By the definition of residual $\boldsymbol{r}_k$ in \eqref{eq:constraint_residual}, together with Assumption~\ref{assump:well_configured} and Lemma~\ref{lemma:edge_agreement}, this proves the residual convergence, i.e. $\boldsymbol{x}_k - \boldsymbol{z}_k \to \boldsymbol{0}$ and $\boldsymbol{A}_{ij}(\boldsymbol{x}_{i,k} - \boldsymbol{x}_{j,k})-\boldsymbol{b}_{ij} \to \boldsymbol{0}$ as $k \to \infty$, $\forall (i,j) \in \mathcal{E}$.
This proves the statement (i).

Given the update rule \eqref{eq:admm_original:multiplier} and the definition of residual $\boldsymbol{r}_k$ in \eqref{eq:constraint_residual}, recall that $\boldsymbol{y}_{k+1} = \boldsymbol{y}_k + \rho \boldsymbol{r}_{k+1}$.
Then $\boldsymbol{y}_{k+p} = \boldsymbol{y}_k + \rho (\boldsymbol{r}_{k+1}+\cdots+\boldsymbol{r}_{k+p}), \ p \geq 2$.
Consequently,
\begin{equation*}
\norm{\boldsymbol{y}_{k+p} - \boldsymbol{y}_k} \leq \rho(\norm{\boldsymbol{r}_{k+1}} + \cdots + \norm{\boldsymbol{r}_{k+p}}).
\end{equation*}
Recall that the series $\sum_{k=0}^{\infty} \boldsymbol{r}_k$ converges (and is a Cauchy sequence), then for any $\epsilon >0$, one can find a positive integer $N$ such that
\begin{equation*}
\rho (\norm{\boldsymbol{r}_{k+1}} + \cdots + \norm{\boldsymbol{r}_{k+p}}) < \epsilon, \ \forall k,p+k \geq N.
\end{equation*}
Thus, the sequence $\boldsymbol{y}_k$ is also a Cauchy sequence, thus it converges to a point, denoted as $\tilde{\boldsymbol{y}}$.
Similarly, due to the series $\sum_{k=0}^{\infty} \boldsymbol{D}(\boldsymbol{z}_{k+1} - \boldsymbol{z}_k)$ converges, one deduce that the sequence $\boldsymbol{D}\boldsymbol{z}_k$ converges.
By definition \eqref{eq:constraint_residual}, since $\boldsymbol{C} \boldsymbol{x}_k + \boldsymbol{D} \boldsymbol{z}_k - \boldsymbol{E} = \boldsymbol{r}_k$, the convergence of $\boldsymbol{r}_k$ and $\boldsymbol{D}(\boldsymbol{z}_{k+1} - \boldsymbol{z}_k)$ implies that the sequence $\boldsymbol{C}\boldsymbol{x}_k$ also converges.

Consider the inequality $\boldsymbol{\ell}^* - \boldsymbol{\ell}_{k+1} \leq {\boldsymbol{y}^*}^{\top} \boldsymbol{r}_{k+1}$ from Lemma~\ref{lemma:inequality_1}, the right-hand side goes to zeros as $k \to \infty$ because $\boldsymbol{r}_k \to \boldsymbol{0}$.
Consider the inequality \eqref{eq:lemma_inequality_2} from Lemma~\ref{lemma:inequality_2}, since $\boldsymbol{D}(\boldsymbol{z}_{k+1} - \boldsymbol{z}^*)$ is bounded, both $\boldsymbol{r}_{k+1}$ and $\boldsymbol{D}(\boldsymbol{z}_{k+1}-\boldsymbol{z}_k)$ go to zeros as $k \to \infty$, and the sequence $\boldsymbol{y}_{k+1}$ converges, the right-hand side of \eqref{eq:lemma_inequality_2} goes to zeros as $k \to \infty$.
Thus, $\lim_{k \to \infty} \boldsymbol{\ell}_k = \boldsymbol{\ell}^*$, which proves the objective convergence, i.e. the statement (ii).

Next, one needs to prove that the sequences $\boldsymbol{x}_k$ and $\boldsymbol{z}_k$ converge.
Note that $\boldsymbol{C}^\top \boldsymbol{C}$ and $\boldsymbol{D}^\top \boldsymbol{D}$ are invertible.
Since $\boldsymbol{C}\boldsymbol{x}_k$ and $\boldsymbol{D}\boldsymbol{z}_k$ converge to some values, $\boldsymbol{x}_k$ and $\boldsymbol{z}_k$ converge to their corresponding value multiplied by $(\boldsymbol{C}^\top \boldsymbol{C})^{-1}\boldsymbol{C}^\top$ and $(\boldsymbol{D}^\top \boldsymbol{D})^{-1}\boldsymbol{D}^\top$, respectively.
Denote $\boldsymbol{x}_k$ and $\boldsymbol{z}_k$ converge to $\tilde{\boldsymbol{x}}$ and $\tilde{\boldsymbol{z}}$, respectively.
Next is to prove they converge to an optimal solution.

Since for every iteration $k$, $\boldsymbol{C} \boldsymbol{x}_k + \boldsymbol{D} \boldsymbol{z}_k - \boldsymbol{E} = \boldsymbol{r}_k$, there also holds the following for the limit,
\begin{equation} \label{eq:KKT_constraint}
\boldsymbol{C} \tilde{\boldsymbol{x}} + \boldsymbol{D} \tilde{\boldsymbol{z}} - \boldsymbol{E} = \boldsymbol{0}.
\end{equation}
Using \eqref{eq:KKT_x} for the limit, together with the fact that $\boldsymbol{D}(\boldsymbol{z}_{k+1}-\boldsymbol{z}_k)$ converges to zero, there holds
\begin{equation} \label{eq:KKT_x_limit}
\boldsymbol{0} \in \partial \boldsymbol{f}(\tilde{\boldsymbol{x}}) + \boldsymbol{C}^\top \tilde{\boldsymbol{y}}.
\end{equation}
Using \eqref{eq:KKT_z} for the limit, there holds
\begin{equation} \label{eq:KKT_z_limit}
\boldsymbol{0} \in \textstyle \sum_{i=1}^{m} \partial\mathcal{I}_{\mathcal{X}_i}(\tilde{\boldsymbol{z}}_i) + \boldsymbol{D}^\top \tilde{\boldsymbol{y}}.
\end{equation}
Given \eqref{eq:KKT_x_limit} and \eqref{eq:KKT_z_limit}, there holds
\begin{equation} \label{eq:KKT_final}
\boldsymbol{0} \in \partial \boldsymbol{f}(\tilde{\boldsymbol{x}}) + \sum_{i=1}^{m} \partial\mathcal{I}_{\mathcal{X}_i}(\tilde{\boldsymbol{z}}_i) + \boldsymbol{C}^\top \tilde{\boldsymbol{y}} + \boldsymbol{D}^\top \tilde{\boldsymbol{y}},
\end{equation}
Since \eqref{eq:KKT_constraint} and \eqref{eq:KKT_final} are exactly the KKT equations \eqref{eq:KKT_saddle}. By \cite[Theorem~28.3]{Rockafellar+1970}, one conclude that $(\tilde{\boldsymbol{x}}, \tilde{\boldsymbol{z}}, \tilde{\boldsymbol{y}})$ are one saddle point $(\boldsymbol{x}^*, \boldsymbol{z}^*, \boldsymbol{y}^*)$.
Thus, the proposed rule \eqref{eq:distributed_update_2} drives $\boldsymbol{x}_k$ and $\boldsymbol{z}_k$ to an optimal solution $(\boldsymbol{x}^*, \boldsymbol{z}^*)$ of Problem \eqref{problem_interest_revised_2}.
By \cite[Theorem~28.4]{Rockafellar+1970}, $\boldsymbol{y}^*$ is an optimal solution to the dual problem of Problem \eqref{problem_interest_revised_2}.
Together with Lemma~\ref{lemma:convert_problem}, it proves the statement (iii) and (iv). This completes the proof.
\end{proof}

\subsection{Numerical Simulation} \label{subsec:sim}

This subsection presents a numerical simulation to verify the proposed update rule \eqref{eq:distributed_update_2}.
Suppose there are $m=4$ agents. Denote $\boldsymbol{x}_i \in \mathbb{R}^2$ as the position for agent $i$, which is randomly initialized within $\mathcal{X}_i = {[-100, 100]}^2 \subset \mathbb{R}^2$.
Agents share and update their states to cooperatively minimize a global objective function by individually minimizing their local objective functions.
Also, agents should eventually achieve some desired edge agreements \eqref{example:desired_edge_agree}.

The local objective functions are defined as follows:
\begin{equation}
\begin{split}
&\boldsymbol{f}_1(\boldsymbol{x}_1) = \norm{\boldsymbol{x}_1}^2, \ \boldsymbol{f}_2(\boldsymbol{x}_2) = \norm{\boldsymbol{x}_2 - \matt{2 & 2}^{\top}}^2, \\
& \boldsymbol{f}_3(\boldsymbol{x}_3) = \norm{\boldsymbol{x}_2 + \matt{3 & 3}^{\top}}^2, \ \boldsymbol{f}_4(\boldsymbol{x}_4) = \textstyle\sum_{i=1}^2 e^{\boldsymbol{x}_4[i]},
\end{split}
\end{equation}
where $\boldsymbol{x}_4[i] \in \mathbb{R}$ denotes the $i$-th entry of vector $\boldsymbol{x}_4$.
Note that $\boldsymbol{f}(\boldsymbol{x}) = \sum_{i=1}^m \boldsymbol{f}_i(\boldsymbol{x}_i)$ is strictly convex in $\boldsymbol{x}$.

The network communication topology is shown in Fig. \ref{fig:network_simple} with the edge set $\mathcal{E}=\{ (1,2), (2,3), (3,1), (3,4) \}$.
\begin{figure}[ht]
\centering
\includegraphics[width=0.15\textwidth]{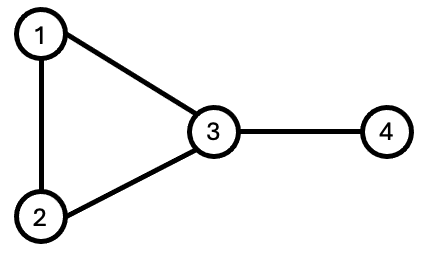}
\caption{The network communication topology.}
\label{fig:network_simple}
\end{figure}
The desired edge agreements are
\begin{equation} \label{example:desired_edge_agree}
\begin{split}
& \boldsymbol{A}_{ij}(\boldsymbol{x}_{i} - \boldsymbol{x}_{j}) = \boldsymbol{b}_{ij}, \ \boldsymbol{A}_{ij} = \boldsymbol{I}_2, \ \forall (i,j) \in \mathcal{E}, \\
&\boldsymbol{b}_{12} = -\boldsymbol{b}_{21} = \matt{0 & 3}^{\top}, \boldsymbol{b}_{23} = -\boldsymbol{b}_{32} = \matt{-2.6 & -1.5}^{\top}, \\
&\boldsymbol{b}_{31} = -\boldsymbol{b}_{13} = \matt{2.6 & -1.5}^{\top}, \boldsymbol{b}_{34} = -\boldsymbol{b}_{43} = \matt{-3 & 0}^{\top}. \\
\end{split}
\end{equation}
The oriented incidence matrix is
\begin{equation*}
\boldsymbol{H} = \matt{ 1 & -1 & 0 & 0 \\ 0 & 1 & -1 & 0 \\ 1 & 0 & -1 & 0 \\ 0 & 0 & 1 & -1}.
\end{equation*}
Introduce the following index to measure how well the edge agreement constraints are satisfied, 
\begin{equation}
W_{1,k} = \textstyle\sum_{(i,j) \in \mathcal{E}} \norm{\boldsymbol{A}_{ij}(\boldsymbol{x}_{i,k} - \boldsymbol{x}_{j,k})-\boldsymbol{b}_{ij}}^2 \geq 0,
\end{equation}
where $W_{1,k}=0$ if and only if all the edge agreements in \eqref{problem_interest:edge} are satisfied.
The numerical result is obtained by applying the update \eqref{eq:distributed_update_2} with CasADi \cite{andersson2019casadi} and the IPOPT solver \cite{wachter2006implementation}, with $\rho = 5.0$.
As shown in the upper portion of Fig. \ref{fig:index_plot}, the proposed update \eqref{eq:distributed_update_2} drives all agents to satisfy the edge agreements \eqref{example:desired_edge_agree} exponentially.
\begin{figure}[ht]
\centering
\includegraphics[width=0.30\textwidth]{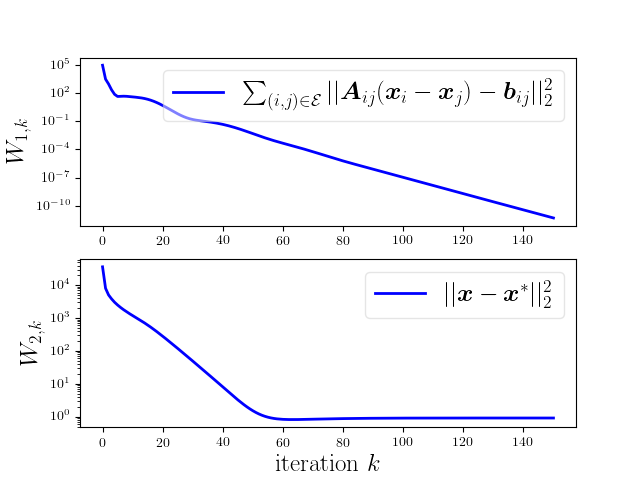}
\caption{A trajectory of $W_{1,k}$ and $W_{2,k}$ over iteration $k$ in log scale.}
\label{fig:index_plot}
\end{figure}

Fig.~\ref{fig:obj_plot} shows the trajectory of $\boldsymbol{f}(\boldsymbol{x}_k)$ over iteration $k$ and the global optimum $\boldsymbol{x}^* \in \mathbb{R}^8$.
$\boldsymbol{x}^*$ is obtained by solving Problem \eqref{problem_interest} with IPOPT in a centralized manner. Define the following index to measure the distance to the global optimum:
\begin{equation}
W_{2,k} = \norm{\boldsymbol{x}_k - \boldsymbol{x}^*}^2 \geq 0,
\end{equation}
where $W_{2,k}=0$ if and only if $\boldsymbol{x}_k = \boldsymbol{x}^*$.
As shown in the lower portion of Fig. \ref{fig:index_plot}, $\boldsymbol{x}_k$ converges to $\boldsymbol{x}^*$ exponentially. The above numerical results validate the proposed algorithm.

\begin{figure}[ht]
\centering
\includegraphics[width=0.30\textwidth]{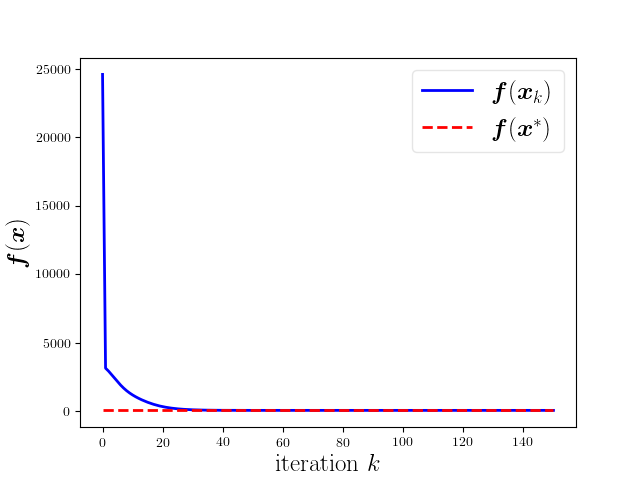}
\caption{A trajectory of $\boldsymbol{f}(\boldsymbol{x}_k)$ and $\boldsymbol{x}^*$ over iteration $k$.}
\label{fig:obj_plot}
\end{figure}

\section{Application: Distributed Battery Network Energy Management} \label{sec:app}

This section presents an application of the proposed theory of distributed optimization under edge agreement.
A distributed battery network energy management problem, as discussed in \cite{fang2016cooperative}, is formulated in the form of distributed MPC and solved using the proposed algorithm.
Numerical simulations are provided to demonstrate the connection between the proposed theory and the application of distributed MPC.


\subsection{Problem of Battery Network Management}

Consider a network of Lithium-ion battery energy storage systems (LiBESSs) that contains $m$ nodes, and $\mathcal{V} \triangleq \{ 1, \cdots, m \}$. For each node $i$, there exists a LiBESS, where the battery's nominal maximum capacity is $Q_{i,\mathrm{max}}$ [kWh]; the real-time capacity at time $t_l$ is $Q_{i,l} \leq Q_{i,\mathrm{max}}$; the state-of-charge (SoC) at time $t_l$ is $s_{i,l} \triangleq Q_{i,l}/Q_{i,\mathrm{max}} \cdot 100\%$.
When a charging/discharging power $\tilde{u}_{i,l}$ [kW] at time $t_l$ is applied to node $i$, the battery's discrete-time dynamics is
\begin{equation*}
s_{i,l+1} = s_{i,l} + \frac{\Delta}{3600 Q_{i,\mathrm{max}}} \tilde{\eta}_i(\text{sgn}(\tilde{u}_{i,l})) \tilde{u}_{i,l},
\end{equation*}
where $\Delta >0$ is a discretization step [sec], $\eta_i >0$ represents the charging/discharging efficiency. Specifically, when charging, $\tilde{u}_{i,l} > 0$, $\tilde{\eta}_i \in (0, 1]$; when discharging, $\tilde{u}_{i,l} \leq 0$, $\tilde{\eta}_i \in [1,\infty)$.
The battery dynamics are nonlinear because $\tilde{\eta}_i$ depends on the sign of $\tilde{u}_{i,l}$, reflecting the hysteresis phenomenon of battery dynamics.

The following procedures can be used to rewrite the dynamics as linear.
First, the net charging/discharging power of node $i$ is
\begin{equation}
\tilde{\eta}_i(\text{sgn}(\tilde{u}_{i,l})) \tilde{u}_{i,l} \equiv \matt{\eta_{\mathrm{c},i} & \eta_{\mathrm{d},i}
} \matt{u_{\mathrm{c},i,l} \\ u_{\mathrm{d},i,l}} \triangleq \boldsymbol{\eta}_i \tilde{\boldsymbol{u}}_{i,l}.
\end{equation}
Here, $\boldsymbol{\eta}_i \in \mathbb{R}^{1\times 2}$, where $\eta_{\mathrm{c},i} \in (0, 1]$ and $\eta_{\mathrm{d},i} \in [1, \infty)$ are the charging and discharging efficiency, respectively.
$u_{\mathrm{c},i,l} \in [0, \overline{u}_i]$ and $u_{\mathrm{d},i,l} \in [\underline{u}_i, 0]$ are the charging and discharging power, respectively.
$\underline{u}_{i} < 0$ and $\overline{u}_{i} > 0$ are the discharging and charging power limits, respectively.
Note that the charge/discharge power output to the network is $u_{\mathrm{c},i,l} + u_{\mathrm{d},i,l}$.

Second, with $\tilde{\boldsymbol{u}}_{i,l} \in \mathbb{R}^2$ being treated as a new control variable, the discrete-time dynamics can be rewritten as
\begin{equation}
s_{i,l+1} = s_{i,l} + \alpha_i \boldsymbol{\eta}_i \tilde{\boldsymbol{u}}_{i,l},
\end{equation}
where $\alpha_i \triangleq \frac{\Delta}{3600 Q_{i,\mathrm{max}}}$.
The lower and upper bounds of SoC for each node $i$ are $\underline{s}_{i} \leq s_{i,l} \leq \overline{s}_{i}$.

The LiBESS network is supposed to deliver/absorb electric power to/from an external system.
Denote the known power demand [kW] at time $t_l$ is $P_l$, where $P_l > 0$ and $P_l < 0$ indicates a power output from the LiBESS to the external system and from the external system to the LiBESS, respectively.
The total output power of the network should reach $P_l$:
\begin{equation} \label{eq:batt:p_sum}
-\textstyle\sum_{i=1}^m (u_{\mathrm{c},i,l} + u_{\mathrm{d},i,l}) = -\textstyle\sum_{i=1}^m \matt{1 & 1} \tilde{\boldsymbol{u}}_{i,l} = P_l, \ \forall l.
\end{equation}

Assume the future power demand is known, a cooperative battery network energy management problem at each time step can be formulated in a centralized MPC fashion as follows:
\begin{mini!}|s|
{\{s_{i,l},\tilde{\boldsymbol{u}}_{i,l}\}}{ \textstyle\sum_{i=1}^{m} \textstyle\sum_{l=0}^{T} r_i \norm{\tilde{\boldsymbol{u}}_{i,l}}^2 \label{batt_cen:obj}}
{\label{batt_cen}}{}
\addConstraint{ s_{i,l+1} = s_{i,l} + \alpha_i \boldsymbol{\eta}_i \tilde{\boldsymbol{u}}_{i,l} \label{batt_cen:dyn}}
\addConstraint{ \underline{s}_{i} \leq s_{i,l} \leq \overline{s}_{i}, \ \tilde{\boldsymbol{u}}_{i,l} \in [0, \overline{u}_i] \times [\underline{u}_i, 0] \label{batt_cen:bound}}
\addConstraint{  -\textstyle\sum_{i=1}^m \matt{1 & 1} \tilde{\boldsymbol{u}}_{i,l} = P_l \label{batt_cen:demand} }
\addConstraint{ \forall i \in \mathcal{V}, \ \forall l \in \llbracket 0,T-1 \rrbracket }
\end{mini!}
where the decision variables are $\tilde{\boldsymbol{u}}_{i,l}, \ \forall i \in \mathcal{V}, l \in \llbracket 0,T-1 \rrbracket$ and $s_{i,l}, \ \forall i \in \mathcal{V}, l \in \llbracket 1,T \rrbracket$; $s_{i,0}, \ \forall i \in \mathcal{V}$ is a known initial state; $r_i > 0$ is a prescribed cost parameter.
Note that the time index $l$ in this section differs from the iteration index $k$ used in the previous sections. This distinction is further clarified in Algorithm~\ref{alg:mpc_dist}.


While solving the centralized MPC problem \eqref{batt_cen} is conceptually simple, it faces significant computational challenges when applied to a large LiBESS network, as the optimization must be solved for each time step.
Therefore, the following subsection reformulates this problem as a distributed optimization problem under specific constraints, allowing it to be addressed within the framework of distributed optimization under edge agreements.

\subsection{Problem Reformulation as Distributed MPC}

This subsection reformulates the problem \eqref{batt_cen} into a distributed MPC problem, structured according to the framework of distributed optimization under edge agreements.
A distributed MPC algorithm is then proposed to solve the reformulated problem.

To locally capture the entire picture of the global network and formulate the power demand constraint, each node creates virtual copies of the control variables for all the other nodes, in addition to its own controls.
The resulting extended control variable of node $i$ is
\begin{equation} \label{eq:control_extended_k}
\boldsymbol{u}_{i,l} \triangleq \text{col} \{ \tilde{\boldsymbol{u}}_{i,l,1}, \cdots, \tilde{\boldsymbol{u}}_{i,l,j}, \cdots, \tilde{\boldsymbol{u}}_{i,l,m} \} \in \mathbb{R}^{2m},
\end{equation}
where $\tilde{\boldsymbol{u}}_{i,l,j} \in \mathbb{R}^2$ indicates a copy of $\tilde{\boldsymbol{u}}_{j,l}$ by node $i$.
Note that $\tilde{\boldsymbol{u}}_{i,l,j}$ is not necessarily equal to $\tilde{\boldsymbol{u}}_{j,l}$ initially. However, all nodes shall eventually reach an agreement on these control copies after coordination and communication.

The power demand \eqref{eq:batt:p_sum} can be rewritten as $-\boldsymbol{1}_{1\times 2m} \boldsymbol{u}_{i,l} = P_l$.
For node $i$, the collection of $\boldsymbol{u}_{i,l}$ within the prediction window $\llbracket 0, T \rrbracket $, i.e. from the current time up to $T$ time steps ahead, is
\begin{equation} \label{eq:control_extended_horizon}
\boldsymbol{u}_i \triangleq \text{col} \{ \boldsymbol{u}_{i,0}, \cdots, \boldsymbol{u}_{i,l}, \cdots, \boldsymbol{u}_{i,T-1} \} \in \mathbb{R}^{2mT}.
\end{equation}
$\boldsymbol{u}_i$ represents node $i$'s replica of all the control inputs for the entire network within the window $\llbracket 0, T \rrbracket $.
Thus, the power demand \eqref{eq:batt:p_sum} can be further written as
\begin{equation} \label{eq:batt:p_demand}
\boldsymbol{E} \boldsymbol{u}_i = \boldsymbol{P}, \ \forall i \in \mathcal{V},
\end{equation}
where $\boldsymbol{E} \triangleq \boldsymbol{I}_T \otimes -\boldsymbol{1}_{1\times 2m} \in \mathbb{R}^{T\times 2mT}$; $\boldsymbol{P} \triangleq \text{col} \{ P_0, \cdots, P_{T-1} \} \in \mathbb{R}^T$ indicates a vector of power demand from the current time up to $T-1$ time steps ahead.
To eventually reach an agreement on $\boldsymbol{u}_i$, there exists a consistency constraint among nodes:
\begin{equation} \label{eq:batt:control_agree}
\boldsymbol{u}_i = \boldsymbol{u}_j, \ \forall (i,j) \in \mathcal{E}.
\end{equation}
Denote node $i$'s predicted state vector within the window as
\begin{equation}
\boldsymbol{x}_i \triangleq \text{col} \{ s_{i,1}, \cdots, s_{i,l}, \cdots, s_{i,T} \} \in \mathbb{R}^{T}.
\end{equation}
Given the state and control constraint \eqref{batt_cen:bound}, two closed convex sets $\mathcal{X}_i \subset \mathbb{R}^{T}$ and $\mathcal{U}_i \subset \mathbb{R}^{2mT}$ can be defined accordingly to represent the constraint on $\boldsymbol{x}_i$ and $\boldsymbol{u}_i$, respectively.
The cost function \eqref{batt_cen:obj} can be rewritten as $\textstyle\sum_{i=1}^m \boldsymbol{u}_i^\top \boldsymbol{R} \boldsymbol{u}_i$,
where $\boldsymbol{R} \triangleq \boldsymbol{I}_T \otimes \diag(r_1,r_1,r_2,r_2,\cdots,r_m,r_m) \in \mathbb{R}^{2mT\times 2mT}$.


The system dynamics constraint \eqref{batt_cen:dyn} can be rewritten as
\begin{equation} \label{eq:batt:dyn_all}
\boldsymbol{A}_i \boldsymbol{x}_i + \boldsymbol{B}_i \boldsymbol{u}_i = \boldsymbol{C}_i,
\end{equation}
where $\boldsymbol{A}_i \triangleq \matt{1 & 0 & \cdots & 0 \\ -1 & 1 & \cdots & 0 \\ & \ddots & \ddots & \\ 0 & \cdots & -1 & 1} \in \mathbb{R}^{T\times T}$ is a zero matrix except the ones in the main diagonals and the negative ones in the sub-diagonals;
$\boldsymbol{B}_i \triangleq \boldsymbol{I}_T \otimes -\alpha_i \boldsymbol{\eta}_i \boldsymbol{e}_{i}^\top \in \mathbb{R}^{T\times 2mT}$, where $\boldsymbol{e}_{i} \in \mathbb{R}^{2m \times 2}$ is a matrix with $m$ 2-by-2 vertical blocks; $\boldsymbol{e}_{i}$ contains only one $\boldsymbol{I}_2$ in the $i$-th block and $\boldsymbol{0}_{2\times 2}$ otherwise;
$\boldsymbol{C}_i \triangleq \text{col}\{ s_{i,0} , \boldsymbol{0}_{T-1}\} \in \mathbb{R}^{T}$.

For each node $i$, the decision variables are $\boldsymbol{x}_i$ and $\boldsymbol{u}_i$. Thus, one can concatenate them vertically as
\begin{equation} \label{eq:define_xi_batt}
\boldsymbol{\xi}_i \triangleq \text{col}\{ \boldsymbol{x}_i, \boldsymbol{u}_i \} \in \mathbb{R}^{(2m+1)T} \subset \Xi_i,
\end{equation}
where $\Xi_i \triangleq \mathcal{X}_i \times \mathcal{U}_i$.
Accordingly, the constraints \eqref{eq:batt:p_demand} and \eqref{eq:batt:dyn_all} can be concatenated as
\begin{equation}
\overline{\boldsymbol{A}}_i \boldsymbol{\xi}_i = \overline{\boldsymbol{B}}_i,
\end{equation}
where $\overline{\boldsymbol{A}}_i \triangleq \matt{\boldsymbol{A}_i & \boldsymbol{B}_i \\ \boldsymbol{0}_{T\times T} & \boldsymbol{E}} \in \mathbb{R}^{2T \times (2m+1)T}$, and $\overline{\boldsymbol{B}}_i \triangleq \text{col}\{ \boldsymbol{C}_i, \boldsymbol{P} \} \in \mathbb{R}^{2T}$.
The cost function \eqref{batt_cen:obj} can be further written as $\textstyle\sum_{i=1}^m \boldsymbol{\xi}_i^\top \overline{\boldsymbol{R}} \boldsymbol{\xi}_i$,
where $\overline{\boldsymbol{R}} \triangleq \diag(\boldsymbol{0}_{T\times T}, \boldsymbol{R}) \in \mathbb{R}^{(2m+1)T \times (2m+1)T}$.
The control consistency constraint \eqref{eq:batt:control_agree} can be written as
\begin{equation} \label{eq:control_consistency}
\boldsymbol{A}_{ij} (\boldsymbol{\xi}_i - \boldsymbol{\xi}_j) = \boldsymbol{b}_{ij} \equiv \boldsymbol{0}_{mT}, \ \forall (i,j) \in \mathcal{E},
\end{equation}
where $\boldsymbol{A}_{ij} = \matt{ \boldsymbol{0}_{2mT\times T} & \boldsymbol{I}_{2mT} } \in \mathbb{R}^{2mT \times (2m+1)T}$. Thus,
\begin{equation}
\boldsymbol{P}_{ij} (\boldsymbol{\xi}_i-\boldsymbol{\xi}_j-\bar{\boldsymbol{b}}_{ij}) = \boldsymbol{0}, \ \forall (i,j) \in \mathcal{E},
\end{equation}
where $\boldsymbol{P}_{ij} = \boldsymbol{A}_{ij}^{\top} (\boldsymbol{A}_{ij}\boldsymbol{A}_{ij}^{\top})^{-1}\boldsymbol{A}_{ij} = \matt{ \boldsymbol{0}_{T\times T} & \boldsymbol{0}_{T\times 2mT} \\ \boldsymbol{0}_{2mT\times T} & \boldsymbol{I}_{2mT} } \in \mathbb{R}^{(2m+1)T \times (2m+1)T}$, and $\bar{\boldsymbol{b}}_{ij} = \boldsymbol{A}_{ij}^{\top} (\boldsymbol{A}_{ij}\boldsymbol{A}_{ij}^{\top})^{-1} \boldsymbol{b}_{ij} = \boldsymbol{0}_{(2m+1)T}$.

Therefore, the centralized MPC problem \eqref{batt_cen} can be rewritten as the following distributed optimization problem:
\begin{mini!}|s|
{\{\boldsymbol{\xi}_i\}_{i=1}^m}{ \textstyle\sum_{i=1}^m \boldsymbol{\xi}_i^\top \overline{\boldsymbol{R}} \boldsymbol{\xi}_i \label{batt_dist:obj}}
{\label{batt_dist}}{}
\addConstraint{ \boldsymbol{\xi}_i \in \Xi_i, \ \forall i \in \mathcal{V} }
\addConstraint{ \overline{\boldsymbol{A}}_i \boldsymbol{\xi}_i = \overline{\boldsymbol{B}}_i, \ \forall i \in \mathcal{V} \label{batt_dist:linear_const} }
\addConstraint{ \boldsymbol{P}_{ij} (\boldsymbol{\xi}_i-\boldsymbol{\xi}_j-\bar{\boldsymbol{b}}_{ij}) = \boldsymbol{0}, \ \forall (i,j) \in \mathcal{E}. }
\end{mini!}
As discussed in Remark~\ref{remark:linear_con}, the intersection of the convex set $\Xi_i$ and the hyperplane defined by the linear constraint \eqref{batt_dist:linear_const} is closed and convex.
Thus, define the new local closed and convex set for each node $i$ as
\begin{equation}
\bar{\Xi}_i \triangleq \{ \boldsymbol{z}_i \in \Xi_i \ \vert \ \overline{\boldsymbol{A}}_i \boldsymbol{z}_i - \overline{\boldsymbol{B}}_i = \boldsymbol{0}_{2T} \}.
\end{equation}
Then the problem \eqref{batt_dist} is equivalent to
\begin{mini!}|s|
{\{\boldsymbol{\xi}_i\}_{i=1}^m}{ \textstyle\sum_{i=1}^m \boldsymbol{\xi}_i^\top \overline{\boldsymbol{R}} \boldsymbol{\xi}_i \label{batt_dist_2:obj}}
{\label{batt_dist_2}}{}
\addConstraint{ \boldsymbol{\xi}_i \in \bar{\Xi}_i, \ \forall i \in \mathcal{V} }
\addConstraint{ \boldsymbol{P}_{ij} (\boldsymbol{\xi}_i-\boldsymbol{\xi}_j-\bar{\boldsymbol{b}}_{ij}) = \boldsymbol{0}, \ \forall (i,j) \in \mathcal{E}. \label{batt_dist_2:edge} }
\end{mini!}
\begin{remark}\label{remark:equivalent_problem}
The objective function \eqref{batt_dist_2:obj} is a convex function of decision variables $\boldsymbol{\xi}_i$.
The constraint \eqref{batt_dist_2:edge} or its equivalent \eqref{eq:control_consistency} is essentially the edge agreement constraint \eqref{problem_interest:edge}.
Thus, the problem \eqref{batt_dist_2} is essentially the same as the problem of interest \eqref{problem_interest}.
Therefore, one can adopt Algorithm~\ref{alg:algo_proposed} to solve the problem \eqref{batt_dist_2}, where the convergence is guaranteed by Theorem~\ref{theorem:update_2}.
\end{remark}


By Lemma~\ref{lemma:convert_problem}, the problem \eqref{batt_dist_2} can be rewritten as
\begin{mini!}|s|
{{\{\boldsymbol{\xi}_i,\boldsymbol{z}_i\}}_{i=1}^m}{ \sum_{i=1}^{m} \boldsymbol{\xi}_i^\top \overline{\boldsymbol{R}} \boldsymbol{\xi}_i + \sum_{i=1}^{m} \mathcal{I}_{\bar{\Xi}_i}(\boldsymbol{z}_i) \label{problem_batt_after_2:obj}}
{\label{problem_batt_after_2}}{}
\addConstraint{ \boldsymbol{\xi} = \boldsymbol{z} \label{problem_batt_after_2:equality}}
\addConstraint{ \bar{\boldsymbol{H}}^{\top}\bar{\boldsymbol{P}}(\bar{\boldsymbol{H}} \boldsymbol{\xi} - \bar{\boldsymbol{b}}) = \boldsymbol{0}_{m(2m+1)T}, \label{problem_batt_after_2:edge}}
\end{mini!}
where $\boldsymbol{\xi} \triangleq \text{col} \{ \boldsymbol{\xi}_1, \cdots, \boldsymbol{\xi}_m \} \in \mathbb{R}^{m(2m+1)T}$, $\boldsymbol{z} \triangleq \text{col} \{ \boldsymbol{z}_1, \cdots, \boldsymbol{z}_m \} \in \mathbb{R}^{m(2m+1)T}$.
Given two penalty parameters $\rho_1, \rho_2 >0$, the augmented Lagrangian $L(\boldsymbol{\xi}, \boldsymbol{z}, \boldsymbol{y})$ of Problem \eqref{problem_batt_after_2} is 
\begin{equation} \label{eq:augmented_lagrangian_z_battery}
\begin{split}
&L(\boldsymbol{\xi}, \boldsymbol{z}, \boldsymbol{y}) \triangleq \sum_{i=1}^{m} \boldsymbol{\xi}_i^\top \overline{\boldsymbol{R}} \boldsymbol{\xi}_i + \sum_{i=1}^{m} \mathcal{I}_{\bar{\Xi}_i}(\boldsymbol{z}_i) + \boldsymbol{\lambda}^{\top}(\boldsymbol{\xi}-\boldsymbol{z}) +\\
& \frac{\rho_1}{2} \norm{\boldsymbol{\xi}-\boldsymbol{z}}^2 + \boldsymbol{\mu}^{\top} \bar{\boldsymbol{H}}^{\top}\bar{\boldsymbol{P}}(\bar{\boldsymbol{H}} \boldsymbol{\xi} - \bar{\boldsymbol{b}}) + \frac{\rho_2}{2} \norm{\bar{\boldsymbol{P}}(\bar{\boldsymbol{H}} \boldsymbol{\xi} - \bar{\boldsymbol{b}})}^2,
\end{split}
\end{equation}
where $\boldsymbol{y} \triangleq \text{col} \{ \boldsymbol{\lambda}, \boldsymbol{\mu} \} \in \mathbb{R}^{2m(2m+1)T}$ is the Lagrangian multiplier,
$\boldsymbol{\lambda} \in \mathbb{R}^{m(2m+1)T}$ for constraint \eqref{problem_batt_after_2:equality},
$\boldsymbol{\mu} \in \mathbb{R}^{m(2m+1)T}$ for constraint \eqref{problem_batt_after_2:edge}.
Then at each iteration $k$, the distributed update rule is
\begin{subequations} \label{eq:distributed_update_batt}
\begin{align}
\begin{split}
&\boldsymbol{\xi}_{i,k+1} = \arg\min_{\boldsymbol{\xi}_i} \{ \boldsymbol{\xi}_i^\top \overline{\boldsymbol{R}} \boldsymbol{\xi}_i + \boldsymbol{\lambda}_{i,k}^{\top}\boldsymbol{\xi}_i + \frac{\rho_1}{2} \norm{\boldsymbol{\xi}_i-\boldsymbol{z}_{i,k}}^2 \\
&+\boldsymbol{\mu}_{i,k}^{\top}\textstyle\sum_{j \in \mathcal{N}_i} \boldsymbol{P}_{ij}(\boldsymbol{\xi}_i - \boldsymbol{\xi}_{j,k} - \bar{\boldsymbol{b}}_{ij}) \\
&+\frac{\rho_2}{2}\textstyle\sum_{j \in \mathcal{N}_i} \norm{\boldsymbol{P}_{ij}(\boldsymbol{\xi}_i - \boldsymbol{\xi}_{j,k} - \bar{\boldsymbol{b}}_{ij})}^2,
\end{split} \\
& \boldsymbol{z}_{i,k+1} = \arg\min_{ \substack{ \boldsymbol{z}_i \in \Xi_i \\ \overline{\boldsymbol{A}}_i \boldsymbol{z}_i - \overline{\boldsymbol{B}}_i = \boldsymbol{0} } } \norm{\boldsymbol{z}_i - (\boldsymbol{\xi}_{i,k+1}+\frac{1}{\rho_1}\boldsymbol{\lambda}_{i,k})}^2,\\
& \boldsymbol{\lambda}_{i,k+1} = \boldsymbol{\lambda}_{i,k} + \boldsymbol{\xi}_{i,k+1} - \boldsymbol{z}_{i,k+1},\\
& \boldsymbol{\mu}_{i,k+1} = \boldsymbol{\mu}_{i,k} + \sum_{j \in \mathcal{N}_i} \boldsymbol{P}_{ij}(\boldsymbol{\xi}_{i,k+1} - \boldsymbol{\xi}_{j,k+1} - \bar{\boldsymbol{b}}_{ij}).
\end{align}
\end{subequations}

The distributed MPC algorithm for battery network energy management is summarized in Algorithm~\ref{alg:mpc_dist}. In Line 4, the initial value for each node $i$ is computed by forward propagating node $i$'s dynamics from the current state $s_{i,l}$. During this propagation, only node $i$ contributes to the power demand vector $\boldsymbol{P}$, while the other nodes maintain zero control. This initialization only requires local information of each node $i$.
Alternatively, the solution from the previous time step can be used as a warm-started initial value. In Line 6, node $i$'s control $\tilde{\boldsymbol{u}}^*_{i,0:T-1|l}$ within the prediction horizon $[t_l, t_{l+T-1}]$ can be retrieved using the definitions \eqref{eq:control_extended_k}, \eqref{eq:control_extended_horizon}, and \eqref{eq:define_xi_batt}.
Note that the frequency of executing Line 5 can be lower than $1/\Delta$. For simplicity, this paper adopts the same frequency.

\begin{algorithm}
\caption{Distributed MPC Algorithm}\label{alg:mpc_dist}
\DontPrintSemicolon
\KwIn{$l = 0$, $\rho_1, \rho_2 >0$, $N_{\mathrm{iter}}, T \in \mathbb{Z}_+$, $\Delta > 0$, $P(\cdot)$}
\While {true} {
$\textbf{par}$\For {\emph{node} $i = 1 \ to \ m$} {
$\boldsymbol{P} \gets \matt{P(t_l) & P(t_{l+1}) & \cdots & P(t_{l+T-1})}$\;
Given current state $s_{i,l}$, compute initial value for $\boldsymbol{\xi}_{i,k=0}$, $\boldsymbol{z}_{i,k=0}$, $\boldsymbol{\lambda}_{i,k=0}$, $\boldsymbol{\mu}_{i,k=0}$ \;
$ \boldsymbol{z}_{i}^* \gets$ run Algorithm~\ref{alg:algo_proposed} with update rule \eqref{eq:distributed_update_batt}\;
$\tilde{\boldsymbol{u}}^*_{i,0:T-1|l} \gets$ parse solution $\boldsymbol{z}_{i}^*$\;
$\tilde{\boldsymbol{u}}_{i,l} \gets$ interpolate $\tilde{\boldsymbol{u}}^*_{i,0:T-1|l}$ for time duration $[t_l,t_l + \Delta]$ given zero-order hold\;
perform $\tilde{\boldsymbol{u}}_{i,l}$ until time $t_l + \Delta$\;
}
$t_{l} \gets t_l + \Delta$, $l \gets l+1$\;
}
\end{algorithm}

\subsection{Numerical Results}

This subsection presents the numerical results for solving the distributed battery network energy management problem using Algorithm~\ref{alg:mpc_dist}, the proposed distributed MPC algorithm.
Fig.~\ref{fig:graph_battery} shows the communication topology of a LiBESS network with $m=6$ nodes.
The parameters used for the LiBESS network are given by Table~\ref{table:para_batt}.
The power demand [kW] is given by:
\begin{equation*}
P(t) = 300\text{sin}(0.005\pi t) + 250\text{sin}(0.003\pi t+20).
\end{equation*}
The parameters used in the algorithm are $\rho_1 = 12$, $\rho_2 = 30$, $N_{\mathrm{iter}}=150$. 

The simulation results are shown in Fig.~\ref{fig:sim_batt}. Fig.~\ref{fig:sim_batt:power} illustrates the desired power demand and the actual power output from the network over time, indicating that the network's controls meet the power demand.
Fig.~\ref{fig:sim_batt:soc} depicts the time trajectory of the SoC for all nodes, showing no violation of SoC constraints.
Additionally, Fig.~\ref{fig:sim_batt:power} presents the time trajectory of the controls for all nodes, constrained by their respective lower and upper bounds.
These simulation results validate the effectiveness of the proposed distributed MPC algorithm.

\begin{figure}[ht]
\centering
\includegraphics[width=0.15\textwidth]{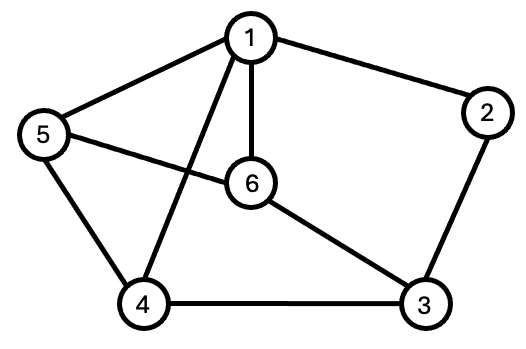}
\caption{Communication topology of a LiBESS network}
\label{fig:graph_battery}
\end{figure}

\begin{table}[ht]
\centering
\begin{threeparttable}
\caption{Parameters of LiBESS network} \label{table:para_batt}
\begin{tabular}{c | c c c c c c}
\toprule
Node $i$ & 1 & 2 & 3 & 4 & 5 & 6 \\
\midrule
$Q_{i,\mathrm{max}}$ [kWh] & 125 & 100  & 80  & 90  & 75  & 200 \\
$\overline{s}_{i}$ [\%]  & 80  & 90   & 90   & 80   & 90   & 80  \\
$\underline{s}_{i}$ [\%] & 30  & 20   & 20   & 30   & 20   & 30  \\
$s_{i,0}$ [\%]           & 50  & 70   & 80   & 80   & 75   & 40  \\
$\overline{u}_{i}$ [kW] & 110 & 100 & 70 & 85 & 60 & 180 \\
$\underline{u}_{i}$ [kW] & -110 & -100 & -70 & -85 & -60 & -180 \\
$r_{i}$                  & 1   & 0.9  & 0.5  & 0.8  & 0.5  & 2   \\
\bottomrule
\end{tabular}
\begin{tablenotes}
\small
\item $\boldsymbol{\eta}_i = \matt{0.9 & 1.1}^\top$, $\forall i \in \mathcal{V}$; $T=20$, $\Delta = 5$ s.
\end{tablenotes}
\end{threeparttable}
\centering
\end{table}

\begin{figure*}[ht]
\centering
\subfloat[Desired power demand and actual power output.]
{\label{fig:sim_batt:power} \includegraphics[width=0.32\linewidth]{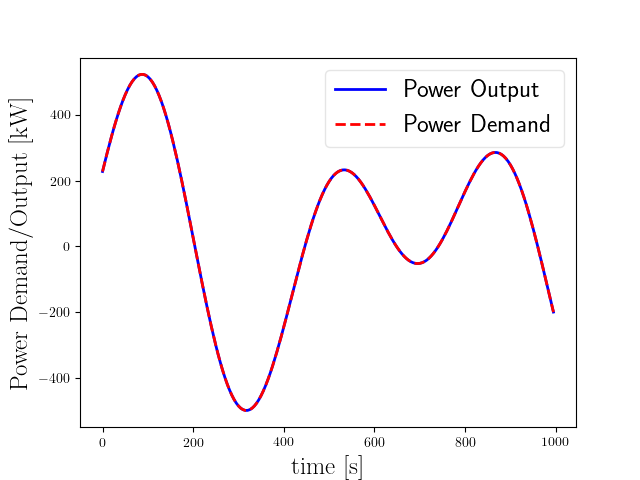}}
\subfloat[SoC trajectories.]
{\label{fig:sim_batt:soc} \includegraphics[width=0.32\linewidth]{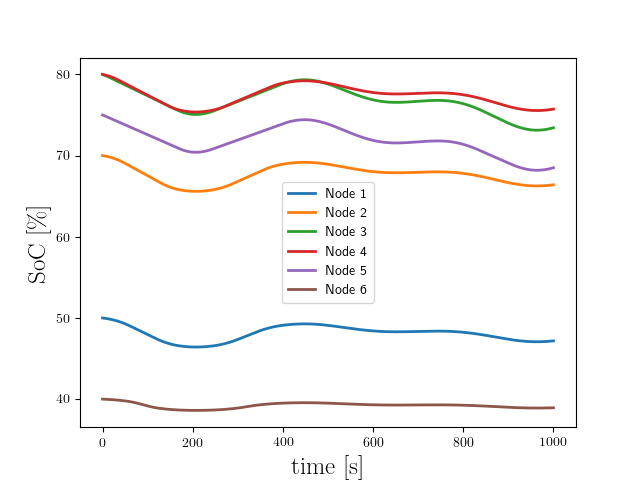}}
\subfloat[Discharging/charging power trajectories.]
{\label{fig:sim_batt:control} \includegraphics[width=0.30\linewidth]{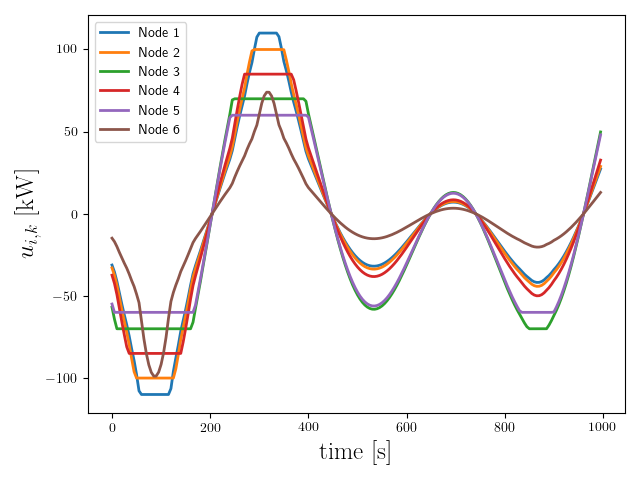}}
\caption{Simulation results of the LiBESS network.} \label{fig:sim_batt}
\end{figure*}

\section{Conclusion} \label{sec:conclusion}

This paper proposes a discrete-time distributed algorithm for solving distributed optimization problems under edge agreements, where each agent is also subject to local convex constraints.
A theoretical analysis is provided to prove the convergence of the proposed algorithm.
Additionally, exemplified by a distributed battery network energy management problem, this paper illustrates the connection between the theory of distributed optimization under edge agreements and distributed MPC. This theoretical framework offers a new perspective to formulate and solve network control or optimization problems, such as those in power grids, battery energy storage systems, autonomous vehicle platooning, and robot cooperative manipulation.

\bibliographystyle{IEEEtran}
\bibliography{Reference}

\end{document}